\documentclass[12pt,reqno]{amsart}

\usepackage{amsmath,amssymb,ifthen}
\usepackage{fullpage}
\usepackage{graphicx,psfrag,subfigure}
\usepackage{color}
\usepackage{moreverb}
\usepackage{amsthm}
\usepackage{mathrsfs}
\usepackage{hhline}
\usepackage{bbm}
\usepackage[english]{babel}
\usepackage{tikz}
\usetikzlibrary{calc,math}
\usepackage{pgfplots}
\usepackage[utf8]{inputenc}
\usepackage[T1]{fontenc}
\usepackage{amsfonts}
\usepackage{enumerate}
\usepackage{hyperref}

\newcommand\coarse{\bullet}
\newcommand\fine{\circ}

\renewcommand{\d}{\,{\rm d}} 
\DeclareMathOperator{\diam}{diam}

\newcommand{\const}[1]{C_{\text{\rm#1}}}

\newcommand{\set}[2]{\big\{#1\,:\,#2\big\}}

\newcommand{\dual}[3][]{#1\langle#2\,,\,#3#1\rangle}
\newcommand{\norm}[3][]{#1\|#2#1\|_{#3}}

\newcommand\refine{{\tt refine}}

\newcommand\C{\mathbb{C}}

\newcommand\N{\mathbb{N}}
\newcommand\R{\mathbb{R}}
\newcommand\T{\mathbb{T}}

\newcommand\MM{\mathcal M}
\newcommand\OO{{\mathcal O}}
\newcommand\TT{\mathcal T}

\numberwithin{equation}{section}
\numberwithin{figure}{section}
\newtheorem{theorem}{Theorem}[section]
\newtheorem{proposition}[theorem]{Proposition}

\newtheorem{corollary}[theorem]{Corollary}
\newtheorem{algorithm}[theorem]{Algorithm}

\newtheorem{remark}[theorem]{Remark}

\newcommand*\patchAmsMathEnvironmentForLineno[1]{%
  \expandafter\let\csname old#1\expandafter\endcsname\csname #1\endcsname
  \expandafter\let\csname oldend#1\expandafter\endcsname\csname end#1\endcsname
  \renewenvironment{#1}%
     {\linenomath\csname old#1\endcsname}%
     {\csname oldend#1\endcsname\endlinenomath}}%
\newcommand*\patchBothAmsMathEnvironmentsForLineno[1]{%
  \patchAmsMathEnvironmentForLineno{#1}%
  \patchAmsMathEnvironmentForLineno{#1*}}%
\AtBeginDocument{%
\patchBothAmsMathEnvironmentsForLineno{equation}%
\patchBothAmsMathEnvironmentsForLineno{align}%
\patchBothAmsMathEnvironmentsForLineno{flalign}%
\patchBothAmsMathEnvironmentsForLineno{alignat}%
\patchBothAmsMathEnvironmentsForLineno{gather}%
\patchBothAmsMathEnvironmentsForLineno{multline}%
}
\usepackage[mathlines]{lineno}

\usepgfplotslibrary{groupplots}
\usepackage{pgfplotstable}

\newcommand{\LL}{\mathcal{L}}
\newcommand{\NN}{\mathcal{N}}
\newcommand{\nnu}{{\overline{\nu}}}
\newcommand{\PP}{\mathcal{P}}
\newcommand{\RR}{\mathcal{R}}
\newcommand{\UU}{\mathcal{U}}

\newcommand{\scal}{\alpha}

\usepackage{cancel}

\title{Adaptive boundary element methods for \\ regularized combined field integral equations}

\author{Th\'eophile Chaumont-Frelet}
\address{Inria Univ. Lille and Laboratoire Paul Painlev\'e, 59655 Villeneuve-d’Ascq, France}
\email{theophile.chaumont@inria.fr}%
\author{Gregor Gantner}
\address{Institute for Numerical Simulation, University of Bonn, Friedrich-Hirzebruch-Allee~7, 53115 Bonn, Germany}
\email{gantner@ins-uni.bonn.de}

\begin{document}

\begin{abstract}
While the exterior Helmholtz problem with Dirichlet boundary conditions is always well-posed,
the associated standard boundary integral equations are not if the squared wavenumber agrees
with an eigenvalue of the interior Dirichlet problem. Combined field integral equations are not
affected by this spurious resonances but are essentially restricted to sufficiently
smooth boundaries. For general Lipschitz domains, the latter integral equations are applicable
through suitable regularization. Under fairly general assumptions on the regularizing operator, 
we propose {\sl a posteriori} computable error estimators for
corresponding Galerkin boundary element methods of arbitrary polynomial degree.
We show that adaptive mesh-refining algorithms steered by these local estimators
converge at optimal algebraic rate with respect to the number of underlying
boundary mesh elements.
In particular, we consider mixed formulations involving the inverse Laplace--Beltrami as
regularizing operator. Numerical examples highlight that in the vicinity of spurious
resonances the proposed adaptive algorithm is significantly more performant when applied to the
regularized combined field equation rather than the standard one.
\end{abstract}

\date{\today}
\maketitle

\section{Introduction}

\begin{subequations}
\label{eq:helmholtz_intro}

\subsection{Model problem}
The propagation of waves is of central importance in a large array of applications
in physics and engineering. Many situations of interest may be modeled assuming a
fixed frequency $k >0$, and considering the reflections of an incident field $u$
on a bounded obstacle $\Omega \subset \mathbb R^d$, $d\in\{2,3\}$, surrounded by an infinite
homogeneous medium $\Omega^{\rm ext}$. In this case, reflected waves can be modeled
by a complex-valued amplitude $U: \Omega^{\rm ext} \to \mathbb C$ satisfying
the Helmholtz equation
\begin{equation}
\Delta U + k^2 U = 0 \text{ in } \Omega^{\rm ext},
\end{equation}
the condition
\begin{equation}
U = u \text{ on } \Gamma
\end{equation}
on the compact boundary $\Gamma := \partial\Omega$ of the obstacle $\Omega$,
and the Sommerfeld radiation condition
\begin{equation}
\partial_r U-ikU = o(r^{(1-d)/2}) \text{ as } r \to \infty,
\end{equation}
where $r$ denotes the radial coordinate, and $\partial_r$ the radial derivative.
\end{subequations}

\subsection{Standard boundary element methods}
For general obstacles $\Omega$, the solution to~\eqref{eq:helmholtz_intro} cannot
be found analytically, and needs to be approximated numerically. In the setting
considered here, boundary element methods (BEM) are particularly attractive,
since they can be directly formulated on the boundary $\Gamma$ and naturally
handle the Sommerfeld radiation condition~\cite{mclean00,steinbach08,ss11}. 
In this work, we consider adaptive
BEM which, given an initial triangulation $\TT_0$ of the boundary, allows for an
automatically steered localized mesh refinement tailored to the solution.
This approach results in accurate and reliable 
approximations potentially requiring much less degrees of freedom than uniform meshes.

To employ a BEM discretization, a natural \emph{direct} approach is to consider the first-kind
integral equation
\begin{equation}
\label{eq:bie_first_kind}
V_k \phi = (K_k-1/2) u,
\end{equation}
where $V_k$ and $K_k$ are single- and double-layer operator for the Helmholtz equation
(see, e.g.,~\cite{mclean00,steinbach08,ss11} and 
Section~\ref{sec:bio} below), and the unknown $\phi$
is the Neumann trace (i.e., the normal derivative on $\Gamma$) of $U$.
The unknown $\phi$ lies in the Sobolev space $H^{-1/2}(\Gamma)$,
and~\eqref{eq:bie_first_kind} is understood in the corresponding dual space $H^{1/2}(\Gamma)$,
leading to a convenient variational framework
(see Section~\ref{sec:sobolev} for a precise definition of $H^{\pm1/2}(\Gamma)$).
Once $\phi$ (or an approximation thereof) is computed, the wave field may be evaluated
in $\Omega^{\rm ext}$ through the representation formula
\begin{align}\label{eq:representation_formula}
U = \widetilde K_k u-\widetilde V_k\phi,
\end{align}
where $\widetilde V_k$ and $\widetilde K_k$
are the single- and double-layer potential of the Helmholtz equation
(see again~\cite{mclean00,steinbach08,ss11} and Section~\ref{sec:bio}).

Similarly, an \emph{indirect} approach is given by the ansatz $U=\widetilde{V}_k\phi$ for some $\phi$ in $H^{-1/2}(\Gamma)$ (which is in general different from $\partial_\nu U$) with corresponding boundary integral equation 
\begin{align}\label{eq:bie_first_kind_indirect}
V_k\phi = u.
\end{align}
Both equations \eqref{eq:bie_first_kind} and \eqref{eq:bie_first_kind_indirect}
involve the single-layer operator $V_k$, and a corresponding adaptive BEM is
formulated and analyzed in~\cite{bbhp19}.

\subsection{Spurious resonances}
However, a significant issue arises when using the first-kind boundary integral
equations~\eqref{eq:bie_first_kind} and \eqref{eq:bie_first_kind_indirect}. 
Indeed, although the original Helmholtz problem in~\eqref{eq:helmholtz_intro}
is well-posed for arbitrary frequencies $k>0$, the operator $V_k$ fails to be
invertible if $k$ is a resonant frequency of $\Omega$ (i.e., if $k^2$ is
a Dirichlet eigenvalue of the Laplace operator on $\Omega$).
This phenomenon is called \emph{spurious resonance} in the literature.
Intuitively, it comes from the fact that~\eqref{eq:bie_first_kind} and
\eqref{eq:bie_first_kind_indirect} simultaneously encompass the exterior
and the interior Dirichlet problem (with corresponding integral equation 
 $V_k \phi = (K_k+1/2) u$ or $V_k\phi=u$). While the former are always well-posed,
the latter are not for resonant frequencies. Indeed, the (finite-dimensional) kernel
of $V_k$ is spanned by the Neumann traces of Dirichlet eigenfunctions with eigenvalue
$k^2$ in $\Omega$; see, e.g., \cite[Theorem~3.9.1]{ss11}.

This leads to many issues in the numerical discretization at
and around spurious resonances, such as ill-conditioning
of the resulting matrices,
slow convergence of iterative linear solvers, and loss of accuracy of the resulting approximation,
cf.\ the numerical experiments of, e.g., \cite{es07,meury07}. Closer to the present topic,
our numerical experiments in Section~\ref{sec:numerics} demonstrate that the convergence of
adaptive BEM is also severely impaired in the vicinity of spurious resonances.
These issues arise both for the direct formulation~\eqref{eq:bie_first_kind},
but also for the corresponding indirect formulation~\eqref{eq:bie_first_kind_indirect}.

\subsection{Combined field integral equations}
These shortcomings have motivated the development of alternative integral
equations to be used in place of~\eqref{eq:bie_first_kind} and~\eqref{eq:bie_first_kind_indirect}.
The seminal works~\cite{bw65,leis65,panich65} simultaneously proposed the indirect ansatz
$U = \widetilde V_k \phi + i \mu \widetilde K_k \phi$ with real user-defined parameter
$\mu\neq 0$ (which typically scales as $k^{-1}$) and corresponding integral equation
\begin{align}\label{eq:bie_combined_field_indirect}
V_k \phi + i \mu (K_k+1/2) \phi = u.
\end{align}
The key idea is that an impedance boundary condition appears for the interior problem,
which is thus always solvable. Similarly, \cite{bm71} proposed the direct approach
\begin{equation}
\label{eq:bie_combined_field}
V_k \phi + i\mu(K_k'+1/2)\phi = (K_k - 1/2)u - i\mu W_k u,
\end{equation}
where $K_k'$ is the adjoint double-layer operator, 
$W_k$ the hypersingular operator,
and the unknown $\phi$ is again the Neumann trace of $U$. It results from exploiting the
representation formula~\eqref{eq:representation_formula} for the sum
$U|_\Gamma + i\mu \partial_\nu U$. 
The equations~\eqref{eq:bie_combined_field_indirect} and~\eqref{eq:bie_combined_field} are
typically called \emph{combined field} integral equations. While this seems not to be widely appreciated in the literature,
this is indeed true for general \emph{Lipschitz} domains; see~\cite[Section~2]{cl07}. 
However, only if the boundary $\Gamma$ is smooth, the operators $K_k, K_k'$ are compact so that the combined integral operators are indeed coercive,
which allows for a feasible numerical discretization. 

In this work, we therefore consider \emph{regularized} combined field integral equations,
which were already introduced in~\cite{panich65} (for theoretical purposes though). 
This regularization consists in applying a regularizing
operator $M: H^{-1/2}(\Gamma) \to H^{1/2}(\Gamma)$ to restore the
variational setting of~\eqref{eq:bie_first_kind} and \eqref{eq:bie_first_kind_indirect}.
In other words, we consider the integral equations
\begin{align}\label{eq:bie_regularized_indirect}
V_k \phi + i M(K_k+1/2) \phi = u
\end{align}
and
\begin{equation}
\label{eq:bie_regularized}
V_k \phi + iM(K_k'+1/2)\phi = (K_k - 1/2)u - iM W_k u,
\end{equation}
where the multiplication by $\mu$ in~\eqref{eq:bie_combined_field_indirect}
and \eqref{eq:bie_combined_field} is replaced by an application of $M$ 

Different numerically feasible choices of $M$ have been suggested in the literature~\cite{bh05,bs06,es07,es08,hiptmair03}. 
All of them guarantee that~\eqref{eq:bie_regularized_indirect} and \eqref{eq:bie_regularized} are well-posed in the sense that the
operators $V_k + iM(K_k+1/2)$ and $V_k + i(K_k'+1/2)M$ are \emph{coercive} isomorphisms between $H^{-1/2}(\Gamma)$ and $H^{1/2}(\Gamma)$
for all $k > 0$.

\subsection{Contributions of the manuscript}
Our key contribution is to extend algorithms and analysis proposed
in~\cite{bbhp19} for adaptive BEM of the standard integral
equations~\eqref{eq:bie_first_kind} and \eqref{eq:bie_first_kind_indirect}
to the regularized combined field integral equations~\eqref{eq:bie_regularized_indirect}
and~\eqref{eq:bie_regularized}. Specifically, given an initial triangulation $\TT_0$,
we discretize~the latter equations with discontinuous piecewise polynomials of degree
$p \in \mathbb N_0$ on a sequence of iteratively refined meshes $\TT_\ell$.

For every possible refinement $\TT_\coarse$ of $\TT_0$, we introduce an
{\sl a posteriori} error estimator $\eta_\coarse$ which is reliable and
\emph{weakly} efficient, which is the state of the art for~\eqref{eq:bie_first_kind}
and~\eqref{eq:bie_first_kind_indirect}. This result is stated in
Theorem~\ref{thm:estimator_indirect}. We then combine this error estimator
with the usual D\"orfler marking to an adaptive algorithm that produces
a sequence of refined meshes $\TT_\ell$ and associated approximations $\phi_\ell$.
We show that, if the initial triangulation is sufficiently fine, then the sequence
$\phi_\ell$ converges at optimal rate with respect to the number of mesh elements
$\# \TT_\ell$. This is stated precisely in Theorem~\ref{thm:dirichlet_indirect_cfie}. 

The results mentioned above apply under mild assumptions on $M$, which are
stated precisely in Section~\ref{section_M}. In particular, the assumptions
are satisfied if $M$ is a positive operator that continuously maps
$H^{-1/2}(\Gamma)$ into $H^1(\Gamma)$. Important examples, also considered in~\cite{bh05,bs06},
that satisfy this assumption are the square of the single-layer operator of the Laplace
operator $V_0^2$, the single-layer operator for the bi-Laplacian $\Delta^2$,
and the inverse of $\scal-\Delta_\Gamma$ with $\scal> 0$
and $\Delta_\Gamma$ the Laplace--Beltrami operator.

Our first set of results is established under the assumption that the
operator $M$ is computed exactly. Although this might be a reasonable
assumption for the boundary integral operators mentioned above (up to
quadrature errors), this is not the case for $M = (\scal-\Delta_\Gamma)^{-1}$.
For this reason, if $M = (\scal-\Delta_\Gamma)^{-1}$, we introduce as in \cite{bh05}
an additional variable $f := M\phi$ for~\eqref{eq:bie_regularized_indirect} 
and $f := M[ (K_k'+1/2)\phi + W_k u]$ for \eqref{eq:bie_regularized}, and consider
the resulting mixed formulation.
The variable $\phi$ is discretized as above and $f$ is approximated by $H^1(\Gamma)$-conforming
piecewise polynomials of degree $p+2$.

Our second set of results therefore concerns these mixed formulations.
Specifically, in this case too, we propose a reliable and weakly efficient {\sl a posteriori}
error estimator that leads to an optimally convergent adaptive BEM algorithm. The corresponding
results are reported in Theorems~\ref{thm:estimator_mixed} and~\ref{thm:dirichlet_mixed_cfie}.

We finally present a set of numerical examples where the adaptive BEM based on
the regularized mixed formulations is compared to the standard
approach based on~\eqref{eq:bie_first_kind} and~\eqref{eq:bie_first_kind_indirect}. 
These examples illustrate our theory by highlighting that the proposed estimators for
the regularized mixed formulations are reliable and (weakly) efficient and that the adaptive
BEM converges at optimal rate.
In particular, the quality of the estimator and convergence speed of the adaptive algorithm
are not affected by nearby interior resonances. In contrast, our examples showcase that this
is not the case for the standard formulations~\eqref{eq:bie_first_kind} and~\eqref{eq:bie_first_kind_indirect}.

An important related comment is that although we do not precisely track the dependency on
$k$ in our analysis, all the constants involved remain uniformly bounded whenever
$k$ lies in a bounded interval in $(0,\infty)$. This means that although
the quality of the estimator and performance of the adaptive algorithm may deteriorate
for very small or very large values of $k$ (which is to be expected), it is indeed
not affected by spurious resonances. 
We emphasize that the focus of the present manuscript is on this robustness
of the regularized combined field formulation in the vicinity of
spurious resonances and its superiority as compared to the standard
one in this regime. In particular, we do not try to demonstrate any
robustness in the high-frequency regime.
Nonetheless, we also provide for completeness
{\sl a posteriori} estimates with constant independent of the frequency
provided that the mesh is fine enough. These results are also highlighted by
numerical experiments with high frequencies.

We finally point out that in this work, we only focus on the so-called \emph{sound-soft}
scattering problems where the boundary condition on $\Gamma$ is of Dirichlet type. Indeed,
\emph{sound-hard} scattering problems involving Neumann boundary conditions are easier to
deal with in the context of combined field integral equations, as they do not need any
regularization; see, e.g., \cite{hiptmair03}. As a result, {\sl a posteriori} error estimation
and optimality of the corresponding adaptive BEM follows as for standard boundary integral
equations~\cite{bbhp19}. For completeness, we mention that regularization can still be beneficial 
for Neumann boundary conditions in terms of preconditionning, see, e.g., \cite{ad07,bet12,bwg17}.
However, this is not the focus of the present manuscript.

\subsection{Outline}
The remainder of this work is organized as follows. In Section~\ref{sec:preliminaries}, 
we make the setting precise, set up notation and collect preliminary results from
the literature. Section~\ref{sec:dirichlet_indirect_cfie} deals with the indirect formulation~\eqref{eq:bie_regularized_indirect} 
where the operator $M$ is assumed to be exactly available. We then analyze in Section~\ref{sec:dirichlet_indirect_mixed},
the mixed formulation of~\eqref{eq:bie_regularized_indirect} where the operator $M = (\scal-\Delta_\Gamma)^{-1}$ is approximated.
Section~\ref{sec:direct} then briefly discusses how the results established for
the indirect formulations can be transferred to their direct counterparts.
%
%
Finally, we present a set of numerical examples in Section~\ref{sec:numerics}.

\section{Preliminaries} \label{sec:preliminaries}

\subsection{General notation}

Throughout and without any ambiguity, $|\cdot|$ denotes the absolute value of scalars, the Euclidean norm of vectors in $\R^n$, or the measure of a set in $\R^n$, e.g., the length of an interval or the area of a surface.
The cardinality of a finite set $\mathcal{S}$ is denoted by $\# \mathcal{S}$.
The set of all linear and continuous operators between two normed  spaces $X$ and $Y$ is denoted by $\LL(X,Y)$.

\subsection{Constants}
\label{sec:hidden_constants}

In the remainder of this work, we employ the notation $A \lesssim B$ for real numbers $A,B \geq 0$ to mean that there exists
a constant $C(d,k,\Gamma,M,\TT_0,p)$ depending \emph{at most} on the dimension $d$, the wavenumber $k$, the boundary $\Gamma$, the regularizing operator $M$,
the initial triangulation $\TT_0$ of $\Gamma$,
and the polynomial degree $p$ of the trial functions such that $A \leq C(d,k,\Gamma,M,\TT_0,p) B$.
Throughout the manuscript, all constants $C$ are additionally uniformly bounded in $k$ on bounded intervals. 
In particular, this applies to the constants $\const{rel},\const{eff},\const{lin},\const{opt}$ as well as $\rho_{\rm lin},\theta_{\rm opt}$ from Theorems~\ref{thm:estimator_indirect}, \ref{thm:dirichlet_indirect_cfie}, ~\ref{thm:estimator_mixed} and~\ref{thm:dirichlet_mixed_cfie}. 
Specifically, for all $K > 1$, 
the estimate $A \lesssim B$ holds true uniformly for all $k \in [1/K,K]$.
If both $A \lesssim B$ and $B \lesssim A$, we also write $A \eqsim B$.

\subsection{Domain and boundary}
Throughout this work, we consider a bounded Lipschitz domain $\Omega \subset \mathbb R^d$, $d\in\{2,3\}$; see e.g., \cite[Definition 3.28]{mclean00} for a precise definition.
We assume that the complement $\Omega^{\rm ext} := \mathbb R^d \setminus \overline{\Omega}$ is connected, and we abbreviate the joint boundary $\Gamma:=\partial\Omega=\partial\Omega^{\rm ext}$.

\subsection{Sobolev spaces on the boundary}
\label{sec:sobolev}

For $\sigma\in[0,1]$, we define the Hilbert spaces $H^{\pm\sigma}(\Gamma)$ as in \cite[page~99]{mclean00} by use of Bessel potentials on $\R^{d-1}$ and liftings via bi-Lipschitz mappings that describe $\Gamma$. 
For $\sigma=0$, it holds that $H^0(\Gamma)=L^2(\Gamma)$ with equivalent norms.
We thus may define $\norm{\cdot}{H^0(\Gamma)}:=\norm{\cdot}{L^2(\Gamma)}$.

For $\sigma\in(0,1]$,  any measurable subset $\omega\subseteq\Gamma$, and all
$v \in H^\sigma(\Gamma)$, we define the associated  Sobolev--Slobodeckij norm
\begin{align*}
\norm{v}{H^{\sigma}(\omega)}^2 := \norm{v}{L^2(\omega)}^2 + |v|_{H^{\sigma}(\omega)}^2
\end{align*}
with
\begin{align*}
|v|_{H^{\sigma}(\omega)}^2
:=
\int_\omega\int_\omega\frac{|v(x)-v(y)|^2}{|x-y|^{d-1+2\sigma}}\,dx dy
\end{align*}
if $\sigma \in (0,1)$, and
\begin{align*}
|v|_{H^1(\omega)} := \norm{\nabla_\Gamma v}{L^2(\omega)}.
\end{align*}
It is well-known that $\norm{\cdot}{H^\sigma(\Gamma)}$ provides
an equivalent norm on $H^\sigma(\Gamma)$. Here, $\nabla_\Gamma(\cdot)$
denotes the (weak) surface gradient.

For $\sigma\in(0,1]$, $H^{-\sigma}(\Gamma)$ is a realization of the dual space of $H^{\sigma}(\Gamma)$.
With the extended $L^2$-scalar product $\dual{\cdot}{\cdot}_{L^2(\Gamma)}$, we define an equivalent norm 
\begin{align*}
	\norm{\phi}{H^{-\sigma}(\Gamma)}
	:=\sup_{v\in H^\sigma(\Gamma)\setminus\{0\}} \frac{|\dual{\phi}{v}_{L^2(\Gamma)}|}{\norm{v}{H^\sigma(\Gamma)}}
	\quad \text{for all }\phi \in H^{-\sigma}(\Gamma).
\end{align*}

\subsection{Model problem}
\label{sec:model}

In the remainder of this work, we fix a wavenumber $k > 0$ as well as a Dirichlet datum
\begin{align}
\label{eq:regularity_u}
u \in H^1(\Gamma).
\end{align}
Our model problem is then to find $U: \Omega^{\rm ext} \to \mathbb C$ such that
\begin{equation}
\label{eq:helmholtz_strong}
\left \{
\begin{array}{rcll}
\Delta U + k^2 U &=& 0 & \text{ in } \Omega^{\rm ext},
\\
                U &=& u & \text{ on } \Gamma,
\\
\partial_r U-ikU &=& o(r^{(1-d)/2}) & \text{ as } r \to +\infty,
\end{array}
\right .
\end{equation}
where $r$ denotes the radial coordinate, and $\partial_r$ the radial derivative.
While the weak formulation of~\eqref{eq:helmholtz_strong} only requires the
assumption~$u \in H^{1/2}(\Gamma)$, the additional regularity~\eqref{eq:regularity_u}
is needed for the definition of the considered {\sl a posteriori} error  estimators.
This mild assumption is satisfied in many applications (e.g., plane wave scattering)
and does not preclude $U$ to be highly singular in the vicinity of corners and edges
(if $d=3$) of $\Gamma$.

As alluded to above, instead of directly computing $U$, we will
reformulate~\eqref{eq:helmholtz_strong} into a boundary integral
equation on $\Gamma$.

\subsection{Trace operators}
We denote by $\gamma_0^{\rm int}: H^1(\Omega)\to H^{1/2}(\Gamma)$ the interior trace operator, which coincides for smooth functions with the usual restriction $(\cdot)|_\Gamma$. 
The exterior trace operator $\gamma_0^{\rm ext}: H^1(\widetilde\Omega \setminus \overline\Omega)\to H^{1/2}(\Gamma)$ is defined analogously for any bounded Lipschitz domain $\widetilde\Omega\supset\overline\Omega$ as restriction to $\Gamma$.

With $H^1(\Delta;\Omega):=\set{u\in H^1(\Omega)}{\Delta u \in L^2(\Omega)}$, we denote by $\gamma_1^{\rm int}: H^1(\Delta,\Omega)\to H^{-1/2}(\Gamma)$ the normal derivative, which coincides for smooth functions with the usual normal derivative $\partial_\nu$. 
The exterior normal derivative $\gamma_1^{\rm ext}: H^1(\Delta;\widetilde\Omega \setminus \overline\Omega)\to H^{-1/2}(\Gamma)$ is defined analogously for any bounded Lipschitz domain $\widetilde\Omega\supset\overline\Omega$ as restriction to $\Gamma$. 

If the considered domain is clear from the context, we will simply write $(\cdot)|_\Gamma$ and $\partial_\nu$.  

\subsection{Layer potentials and boundary integral operators}
\label{sec:bio}
For $k>0$, we recall the Helmholtz kernel 
\begin{align*}
	G_k(z) := \frac{i}{4} H_0^{(1)}(k|z|) \quad \text{for } d=2,\qquad 
	G_k(z) := \frac{e^{ik|z|}}{4\pi |z|} \quad \text{for }d=3,
\end{align*}
where $H_0^{(1)}$ is the first-kind Hankel function of order zero. 
We further use the fundamental solution of the Laplace operator
\begin{align*}
	G_0(z) := -\frac{1}{2\pi} \log|z| \quad \text{for }d=2, \qquad G_0(z) := \frac{1}{4\pi |z|} \quad \text{for }d=3.
\end{align*}

For $k\ge 0$ and sufficiently smooth $\phi,u:\Gamma\to \C$, we define the single-layer potential 
\begin{align*}
	(\widetilde{V}_k\phi)(x) := \int_\Gamma G_k(x-y) \phi(y) \d y
	\quad \text{for all }x\in \R^d\setminus\Gamma
\end{align*}
and the double-layer potential
\begin{align*}
	(\widetilde{K}_k u)(x) := \int_\Gamma \partial_{\nu(y)} G_k(x-y) u(y) \d y
	\quad \text{for all }x\in \R^d\setminus\Gamma.
\end{align*}
For any bounded Lipschitz domain $\widetilde\Omega$, they give rise to bounded linear operators 
\begin{align*}
	\widetilde V_k\in \LL(H^{-1/2}(\Gamma),H^1(\widetilde\Omega)) \quad \text{and}\quad \widetilde K_k \in \LL(H^{1/2}(\Gamma),H^1(\widetilde\Omega\setminus\Gamma)).
\end{align*} 
For $\phi \in H^{-1/2}(\Gamma)$ and $v \in H^{1/2}(\Gamma)$, the potentials $\widetilde{V}_k\phi$ and $\widetilde{K}_kv$ are both solutions of the exterior Helmholtz problem~\eqref{eq:helmholtz_strong} (with suitable Dirichlet data). 
The exact solution of~\eqref{eq:helmholtz_strong} can be represented as
\begin{align}\label{eq:representation}
	U = \widetilde{K}_k(U|_\Gamma) - \widetilde{V} _k(\partial_\nu U).
\end{align}

Applying the trace operators defines the single-layer operator 
\begin{align*}
	V_k:=\gamma_0^{\rm int} \widetilde{V}_k=\gamma_0^{\rm ext} \widetilde{V}_k \in \LL( H^{-1/2}(\Gamma), H^{1/2}(\Gamma)),
\end{align*}
the double-layer operator 
\begin{align*}
	K_k:=\gamma_0^{\rm int} \widetilde{K}_k+1/2=\gamma_0^{\rm ext} \widetilde{K}_k-1/2 \in \LL( H^{1/2}(\Gamma), H^{1/2}(\Gamma)),
\end{align*}
the adjoint double-layer operator 
\begin{align*}
	K_k':=\gamma_1^{\rm int} \widetilde{V}_k-1/2=\gamma_1^{\rm int} \widetilde{V}_k+1/2 \in \LL( H^{-1/2}(\Gamma), H^{-1/2}(\Gamma)),
\end{align*} 
and the hypersingular operator 
\begin{align*}
	W_k:=-\gamma_1^{\rm int} \widetilde{K}_k \in \LL( H^{1/2}(\Gamma), H^{-1/2}(\Gamma)).
\end{align*} 
For $k=0$, the additional mapping properties 
\begin{subequations}
\label{eq:mapping_properties}
\begin{align}
	V_k&\in \LL(H^{-1/2+\sigma}(\Gamma) \to H^{1/2+\sigma}(\Gamma)), 
	\\ 
	K_k&\in \LL(H^{1/2+\sigma}(\Gamma),H^{1/2+\sigma}(\Gamma)),
	\\ \label{eq:mapping_properties_K_pr}
	K_k'&\in \LL(H^{-1/2+\sigma}(\Gamma),H^{-1/2+\sigma}(\Gamma)), 
	\\
	W_k&\in \LL(H^{1/2+\sigma}(\Gamma),H^{-1/2+\sigma}(\Gamma))
\end{align}
\end{subequations}
are well-known for all $\sigma\in [-1/2,1/2]$. 
In~\cite{bbhp19,gp22b}, these mapping properties are also verified for all $k>0$ (and the critical cases $\sigma\in \{-1/2,1/2\}$).

The single-layer operator $V_0:H^{-1/2}(\Gamma)\to H^{1/2}(\Gamma)$ is symmetric and elliptic,
provided that $\diam(\Omega)<1$ if $d=2$. If $d=2$ and $\diam(\Omega)\ge1$, this is still true
up to a compact perturbation, i.e., there is a symmetric and elliptic
$V_0^+:H^{-1/2}(\Gamma)\to H^{1/2}(\Gamma)$ such that $V_0 - V_0^+$ is compact.
If $d=3$ or $d=2$ with $\diam(\Omega)<1$, we abbreviate $V_0^+:=V_0$.

The operator $V_k:H^{-1/2}(\Gamma)\to H^{1/2}(\Gamma)$ is in general only elliptic up to the
compact perturbation $V_0^+-V_k$ with kernel given by the Neumann traces of Dirichlet eigenfunctions.

For a more detailed introduction to layer potentials and boundary integral operators,
we refer to the monographs~\cite{mclean00,steinbach08,ss11}.

\subsection{Meshes and refinement}

Let $T_{\rm ref}$ denote the reference element defined by 
\begin{align*}
	T_{\rm ref} := (0,1) \quad \text{for } d=2 
	\quad\text{and}\quad
	T_{\rm ref} := {\rm conv}\{(0,0),(1,0),(0,1)\} \quad \text{for }d=3.
\end{align*}
We consider conforming triangulations $\TT_\coarse$ of $\Gamma$ characterized similarly as in \cite{affkmp17} by the following properties:
\begin{itemize}
\item Each $T\in\TT_\coarse$ is a relative open subset of $\Gamma$, and there exists a bi-Lipschitz parametrization $\gamma_T:\overline{T_{\rm ref}}\to \overline T$ such that $\gamma_T \in C^2(T_{\rm ref},\R^d)$.
\item The union of all elements cover $\Gamma$, i.e., $\Gamma = \bigcup \overline T$.
\item There are no hanging nodes, i.e., for all $T\neq T'\in\TT_\coarse$, the intersection $\overline T\cap \overline{T'}$ is either empty, a joint node (for $d\ge 2$), or a joint face (for $d=3$).
\item For $d=3$, parametrizations of common boundary parts of neighboring elements are compatible, i.e., if $\emptyset \neq \overline T \cap \overline{T'}=\gamma_T(E_{\rm ref}) = \gamma_{T'}(E_{\rm ref}')$ is a joint edge, then $\gamma_T^{-1}\circ\gamma_{T'}:E_{\rm ref}'\to E_{\rm ref}$ is affine.
\end{itemize}
We define the mesh-size function $h_\coarse \in L^\infty(\Gamma)$ by $h_\coarse|_T:=h_T:=|T|^{1/(d-1)}$ for all $T\in\TT_\coarse$. 
Moreover, let $\NN_\coarse$ denote the set of all nodes in $\TT_\coarse$ with corresponding node patches $\omega_\coarse(z):=\overline{\bigcup \set{T\in\TT_\coarse}{z\in \overline T}}$ for $z\in\NN_\coarse$. 

Let $\TT_0$ be a conforming initial triangulation meeting the above requirements.
(Later, we shall need $\TT_0$ to be sufficiently fine; cf.\ Sections~\ref{sec:dirichlet_indirect_cfie:discretization} and \ref{sec:dirichlet_indirect_mixed:discretization}.)
As mesh refinement strategy, we employ the extended 1D bisection from \cite{affkp13} for $d=2$ and newest vertex bisection from, e.g., \cite{stevenson08}, for $d=3$.
Given a conforming triangulation $\TT_\coarse$ and a set of marked elements $\MM_\coarse\subseteq\TT_\coarse$, the call $\TT_\fine = \refine(\TT_\coarse,\MM_\coarse)$ returns the coarsest refinement $\TT_\fine$ of $\TT_\coarse$ such that all $T\in\MM_\coarse$ have been refined. 
We define the set $\refine(\TT_\coarse)$ as the set of all conforming $\TT_\fine$ that can be obtained from $\TT_\coarse$ by a finite number of refinement steps. 
In particular, $\TT_\coarse$ itself is always in the set of its refinements, i.e., $\TT_\coarse\in\refine(\TT_\coarse)$. 
For the initial mesh, we call $\T:=\refine(\TT_0)$ the set of all admissible triangulations. 
The employed refinement strategies guarantee shape-regularity 
\begin{align*}
	h_T \simeq \diam(T) \quad \text{for all } T\in\TT_\coarse\in\T
\end{align*}
and local quasi-uniformity
\begin{align*}
	h_T \simeq h_{T'} \quad \text{for all } T,T'\in\TT_\coarse\in\T \text{ with }\overline T\cap \overline{T'}\neq \emptyset,
\end{align*}
with hidden constants depending only on the initial triangulation $\TT_0$. 
Moreover, for sequences of successively refined triangulations $(\TT_\ell)_{\ell\in\N_0}$ with $\TT_\ell=\refine(\TT_{\ell-1},\MM_{\ell-1})$ for some $\MM_{\ell-1}\subseteq \TT_{\ell-1}$, and $\TT_\ell \in\T$,  there hold for all $\ell\in\N$  the following three crucial properties:
\begin{enumerate}[(1)]
\renewcommand{\theenumi}{R\arabic{enumi}}
	\item Child estimate: \label{item:child} $\#\TT_\ell \le \const{child} \#\TT_{\ell-1}$ with $\const{child}=2$ for $d=2$ and $\const{child}=4$ for $d=3$; 
	\item Overlay estimate: \label{item:overlay} $\exists(\TT_\ell\oplus\TT_\coarse)\in \refine(\TT_\ell)\cap \refine(\TT_\coarse)$ s.t.\ $(\TT_\ell\oplus\TT_\coarse) \le \#\TT_\ell + \# \TT_\coarse - \#\TT_0$;
	\item Closure estimate: \label{item:closure} $\#\TT_\ell - \#\TT_0 \le \const{clo} \sum_{j=0}^{\ell-1} \# \MM_j$ with $\const{clo}\ge1$ depending only on $\TT_0$.
\end{enumerate}
The child estimate~\eqref{item:child} follows from the fact that in one refinement step, each element is refined at most into $2$ and $4$ elements, respectively, see, e.g., \cite{affkp13} for $d=2$ and \cite{stevenson08} for $d=3$. 
A triangulation $\TT_\ell\oplus\TT_\coarse$ with~\eqref{item:overlay} can indeed be chosen as the literal overlay of two meshes, see, e.g., \cite{affkp13} for $d=2$ and~\cite{stevenson07} for $d=3$. 
The closure estimate~\eqref{item:closure} is proved in \cite{affkp13} for $d=2$ and in \cite{bdd04,stevenson08} for $d=3$. 
The latter two references actually require an admissibility condition of $\TT_0$, which was shown to be redundant in~\cite{kpp13}.

We only mention that the results of the present work also apply for other types of meshes,
for instance curvilinear quadrangulations; cf.\ \cite{gp22b,gp22c,bggpv22} for details in
case of the standard integral equation in~\eqref{eq:bie_first_kind}.

\subsection{Discrete spaces}
Let $\TT_\coarse\in\T$ be a given triangulation. 
For any polynomial degree $q \in\N_0$, we define the space of all (discontinuous) $\TT_\coarse$-piecewise polynomials as
\begin{align*}
	\PP^{q}(\TT_\coarse)
	:= \set{\psi_h \in L^\infty(\Gamma)}{\psi_h|_T\circ\gamma_T \text{ is a polynomial of degree }\le q \text{ for all } T\in\TT_\coarse}.
\end{align*}
For $q\in\N$, we define the space of continuous $\TT_\coarse$-piecewise polynomials as $\mathcal{S}^q(\TT_\coarse):= \PP^q(\TT_\coarse) \cap H^1(\Gamma)$. 
Note that 
\begin{align*}
	\PP^{q}(\TT_\coarse) \subset L^2(\Gamma) \subset H^{-1/2}(\Gamma) 
	\text{ for } q \in\N_0
	\quad\text{and}\quad
	\mathcal{S}^q (\TT_\coarse) \subset H^1(\Gamma) \subset H^{1/2}(\Gamma)
	\text{ for } q \in\N.
\end{align*}

For the remainder of this work, we fix a polynomial degree $p \in \mathbb N_0$.
We will consider the space $\PP^p(\TT_\coarse)$ and the space $\PP^p(\TT_\coarse)\times \mathcal{S}^{p+2}(\TT_\coarse)$ as trial space for the discretization of the considered combined field integral equation and its formulation as mixed system, respectively.

We only mention that the results of the present work also apply for splines on quadrangulations; cf.\ \cite{gp22b,gp22c,bggpv22} for details in case of standard integral equations~\eqref{eq:bie_first_kind}. 

\subsection{Inverse estimates}
The following crucial inverse estimates were proved in \cite{affkmp17}
for $k=0$ and generalized in~\cite{bbhp19,gp22b} to arbitrary $k\ge0$:
\begin{subequations}
\label{eq:invest}
\begin{align}\label{eq:invest_V}
	\norm{h_\coarse^{1/2} \nabla_\Gamma V_k\psi}{L^2(\Gamma)}
	&\lesssim
	\norm{\psi}{H^{-1/2}(\Gamma)} + \norm{h_\coarse^{1/2} \psi}{L^2(\Gamma)},
	\\ \label{eq:invest_K}
	\norm{h_\coarse^{1/2} \nabla_\Gamma K_k v}{L^2(\Gamma)}
	&\lesssim
	\norm{v}{H^{1/2}(\Gamma)} + \norm{h_\coarse^{1/2} \nabla_\Gamma v}{L^2(\Gamma)},
	\\ \label{eq:invest_K_pr}
	\norm{h_\coarse^{1/2} K_k'\psi}{L^2(\Gamma)}
	&\lesssim
	\norm{\psi}{H^{-1/2}(\Gamma)} + \norm{h_\coarse^{1/2} \psi}{L^2(\Gamma)},
	\\ \label{eq:invest_W}
	\norm{h_\coarse^{1/2} W_k v}{L^2(\Gamma)}
	&\lesssim
	\norm{v}{H^{1/2}(\Gamma)} + \norm{h_\coarse^{1/2} \nabla_\Gamma v}{L^2(\Gamma)}
\end{align}
\end{subequations}
for all triangulations $\TT_\coarse\in\T$ and all function $\psi \in L^2(\Gamma)$ and $v \in H^{1}(\Gamma)$ with hidden constants that do not depend on the wavenumber $k$.

We further recall the discrete inverse estimates 
\begin{subequations}
\label{eq:discrete_invest}
\begin{align} \label{eq1:discrete_invest}
	\norm{h_\coarse^{1/2}\psi_\coarse}{L^2(\Gamma)} &\lesssim
	\norm{\psi_\coarse}{H^{-1/2}(\Gamma)} 
	\quad\text{for all }\psi_\coarse \in \PP^p(\TT_\coarse),
	\\
	\norm{h_\coarse^{1/2}\nabla_\Gamma v_\coarse}{L^2(\Gamma)} &\lesssim
	\norm{v_\coarse}{H^{1/2}(\Gamma)} 
	\quad\text{for all }v_\coarse \in \mathcal{S}^{p+2}(\TT_\coarse);
\end{align}
\end{subequations}
with hidden constants that do not depend on the polynomial degree $p$;
cf.\ \cite[Lemma~A.1]{affkmp17} and \cite[Proposition~5]{affkp15}, see also the original works~\cite{dfghs04,ghs05}.

\subsection{Regularizing operator}
\label{section_M}

Our regularized combined field integral equations invovle a regularizing operator
$M: H^{-1/2}(\Gamma) \to H^{1/2}(\Gamma)$. Throughout, we assume that 
\begin{align}
\label{eq:compact}
\tag{M1}
M \in \LL(H^{-1/2}(\Gamma), H^{1/2}(\Gamma)) \text{ is compact,}
\end{align} 
that it is positive in the sense that
\begin{align}
\label{eq:positive}
\tag{M2}
{\rm Re} \dual{M\psi}{\psi}_{L^2(\Gamma)} > 0 \quad \text{for all } \psi \in H^{-1/2}(\Gamma) \setminus \{0\},
\end{align}
and that it satisfies the following inverse inequality
\begin{align}
\label{eq:assumption_inverse_new}
\tag{M3}
	\norm{h_\coarse^{1/2} \nabla_\Gamma M\psi}{L^2(\Gamma)}
	\lesssim \norm{\psi}{H^{-1/2}(\Gamma)} + \norm{h_\coarse^{1/2} \psi}{L^2(\Gamma)}
	\quad\text{for all }\psi \in L^2(\Gamma),
\end{align}
which is similar to the inverse inequalities in~\eqref{eq:invest} for the standard integral operators.
Note that \eqref{eq:compact} and \eqref{eq:assumption_inverse_new} particularly imply that
\begin{align}
\label{eq:assumption_h1}
M \in \LL(L^2(\Gamma), H^1(\Gamma)).
\end{align}
Conversely, we point out that~\eqref{eq:compact} and~\eqref{eq:assumption_inverse_new} are both
trivially satisfied provided that $M \in \LL(H^{-1/2}(\Gamma),H^1(\Gamma))$, which is the case
for all concrete examples of the operator $M$ considered in this work.

\begin{proposition}
\label{prop:assumptions}
If $M \in \LL(H^{-1/2}(\Gamma),H^1(\Gamma))$, then~\eqref{eq:compact}
and~\eqref{eq:assumption_inverse_new} hold true.
\end{proposition}

\begin{proof}
The requirement in~\eqref{eq:compact} simply follows from the fact
that the injection $H^1(\Gamma) \hookrightarrow H^{1/2}(\Gamma)$ is compact.
For~\eqref{eq:assumption_inverse_new}, the inequality~$h_\coarse \lesssim \diam(\Gamma)$
and continuity of $M$ show for all $\psi \in H^{-1/2}(\Gamma)$ that
\begin{equation*}
\|h_\coarse^{1/2}\nabla_\Gamma M \psi\|_{L^2(\Gamma)}
\lesssim
\|M \psi\|_{H^1(\Gamma)}
\lesssim
\|\psi\|_{H^{-1/2}(\Gamma)},
\end{equation*}
and trivially,
\begin{equation*}
\|h_\coarse^{1/2}\nabla_\Gamma M \psi\|_{L^2(\Gamma)}
\lesssim
\|\psi\|_{H^{-1/2}(\Gamma)} + \|h_\coarse^{1/2} \psi\|_{L^2(\Gamma)}
\end{equation*}
whenever $\psi \in L^2(\Gamma)$.
\end{proof}

As suggested in~\cite{bs06}, the operator $M$ might be realized with an integral operator.
For instance, we can use the square of the single-layer operator $V_0^2$ (provided that
$\diam(\Omega)<1$ if $d=2$) or the single-layer operator
of the biharmonic operator $(\scal-\Delta)^2$ with $\scal > 0$.
Both of them continuously map $H^{-1}(\Gamma)$ to $H^1(\Gamma)$, 
which implies \eqref{eq:compact} and~\eqref{eq:assumption_inverse_new},
and also satisfy~\eqref{eq:positive}.

Another interesting choice of $M$ is the inverse of
a differential operator. In particular, we will consider, as in~\cite{bh05},
$M := (\scal-\Delta_\Gamma)^{-1}$ with $\scal > 0$. 
Here, $\Delta_\Gamma\in \LL(H^1(\Gamma),H^{-1}(\Gamma))$ is the Laplace--Beltrami operator defined weakly as
\begin{align}
	\dual{\Delta_\Gamma v }{w}_{L^2(\Gamma)}:=-\dual{\nabla_\Gamma v}{\nabla_\Gamma w}_{L^2(\Gamma)}
	\quad\text{for all }v,w\in H^1(\Gamma).
\end{align}
Also this choice of $M$ continuously maps $H^{-1}(\Gamma)$ to $H^1(\Gamma)$.
Moreover, there holds the ellipticity
\begin{align*}
	{\rm Re}\dual{M\psi}{\psi}_{L^2(\Gamma)} \gtrsim \norm{\psi}{H^{-1}(\Gamma)} 
	\quad \text{for all }\psi \in H^{-1}(\Gamma),
\end{align*}
which guarantees positivity~\eqref{eq:positive}.
However, as already mentioned, such an $M$ is not computable (even up to
quadrature errors), and it needs to be approximated. This leads to mixed formulations analyzed
in Sections~\ref{sec:dirichlet_indirect_mixed} and~\ref{sec:direct_mixed}.

Similarly, one could consider, as in~\cite{bh05}, the inverse of the piecewise Laplace--Beltrami
operator for $M$, i.e., $M\psi \in H^1_0(\TT_0)$ with
\begin{align}
\label{eq:pw_laplace_beltrami}
\dual{\nabla_\Gamma M \psi}{\nabla_\Gamma v}_{L^2(\Gamma)} = \dual{\psi}{v}_{L^2(\Gamma)}
\quad\text{for all } \psi\in H^{-1}(\Gamma), v \in H_0^1(\TT_0),
\end{align}
where $H_0^1(\TT_0) := \set{v \in H^1(\Gamma)}{v|_{T_0} \in H_0^1(T)\text{ for all }T_0\in\TT_0}$. 
This allows to exploit the shift theorem on the regular surfaces $T_0\in\TT_0$, being diffeomorphic
images of convex sets.

\begin{remark}
The works \cite{hiptmair03,es07,es08} consider a regularizing operator $M$ that involves the inverse of a \emph{stabilized} hypersingular operator $\widetilde W_0$. 
For general Lipschitz domains, this operator is \emph{not} compact. 
As in Sections~\ref{sec:dirichlet_indirect_mixed} and~\ref{sec:direct_mixed}, where we will consider $M := (\scal-\Delta_\Gamma)^{-1}$, this approach naturally leads to a mixed system. 
Similarly as we do for the Laplace--Beltrami equation, one can use a second weighted-residual estimator for the involved hypersingular integral equation. 
We expect that the analogous results of Sections~\ref{sec:dirichlet_indirect_mixed} and~\ref{sec:direct_mixed} can be proved by combining the analysis of this work 
with the {\sl a posteriori} and convergence analysis for the standard hypersingular integral equation~\cite{ffkmp15}.
\end{remark}

\section{Indirect integral equation} \label{sec:dirichlet_indirect_cfie}

Given a regularizing operator $M\in \LL(H^{-1/2}(\Gamma), H^{1/2}(\Gamma))$
satisfying~\eqref{eq:compact} and~\eqref{eq:positive}, we follow \cite{bh05} and
make the ansatz 
\begin{align}\label{eq:dirichlet_ansatz}
	U = \widetilde V_k \phi + i \widetilde K_k M \phi
\end{align}
for some suitable density function $\phi \in H^{-1/2}(\Gamma)$ to obtain a solution of the exterior Dirichlet problem~\eqref{eq:helmholtz_strong}.
Applying the (exterior) trace operator to~\eqref{eq:dirichlet_ansatz} and exploiting the Dirichlet condition of~\eqref{eq:helmholtz_strong}, we see that 
\begin{align*}
	u = U|_\Gamma = V_k \phi + i (K_k+1/2) M \phi.
\end{align*}
With the short-hand notation
\begin{equation*}
	S:= V_k + i(K_k+1/2) M \in \LL(H^{-1/2}(\Gamma),H^{1/2}(\Gamma)),
\end{equation*}
we hence need to find $\phi \in H^{-1/2}(\Gamma)$ such that
\begin{align}
\label{eq:dirichlet_indirect_cfie}
S\phi = u.
\end{align}

Clearly, the operator $S$ is the sum of the symmetric and elliptic $V_0^+$ and the compact perturbation $V_k-V_0^+ + i(K_k+1/2) M$. 
According to~\cite{bh05}, $S$ is also injective and thus even bijective by the Fredholm alternative. 
In other words, the problem in~\eqref{eq:dirichlet_indirect_cfie} is well-posed for all $u\in H^{1/2}(\Gamma)$.

\subsection{Discretization}
\label{sec:dirichlet_indirect_cfie:discretization}

Given any $\TT_\coarse\in\T$, the discrete problem consists in finding a Galerkin approximation
$\phi_\coarse \in \PP^p(\TT_\coarse)$ of $\phi$
satisfying the discrete variational formulation
\begin{align}\label{eq:dirichlet_indirect_cfie:discretization}
	\dual{S\phi_\coarse}{\psi_\coarse}_{L^2(\Gamma)} = \dual{u}{\psi_\coarse}_{L^2(\Gamma)}
	\quad \text{for all } \psi_\coarse \in \PP^p(\TT_\coarse).
\end{align}

To theoretically guarantee well-posedness of~\eqref{eq:dirichlet_indirect_cfie:discretization},
we assume that the initial triangulation $\TT_0$ 
is sufficiently fine such that 
\begin{align}\label{eq:dirichlet_indirect_cfie:discrete_inf-sup}
	\inf_{\chi_\coarse \in \PP^p(\TT_\coarse)}\sup_{\psi_\coarse \in \PP^p(\TT_\coarse)}  \frac{|\dual{S \chi_\coarse}{\psi_\coarse}_{L^2(\Gamma)}|}{\norm{\psi_\coarse}{H^{-1/2}(\Gamma)}}
	\gtrsim 1
	\quad \text{for all }\TT_\coarse\in\T,
\end{align}
where we make the convention that $\tfrac00:=0$; see also Remark~\ref{rem:vienna_cheat} for a way (which is not fully satisfactory though) to avoid this assumption.

Owing to the fact that $S$ is symmetric and elliptic up to some compact perturbation,
it is well known that there exists some ${\tt h} \gtrsim 1$ such that~\eqref{eq:dirichlet_indirect_cfie:discrete_inf-sup}
is satisfied whenever $h_T \leq {\tt h}$ for all $T \in \TT_0$; see, e.g., \cite[Theorem 4.2.9]{ss11} or \cite[Proposition~1]{bhp17}.
The dependency on ${\tt h}$, $k$, $p$, and $\Gamma$ is hard to track, but we refer the reader to~\cite{lohndorf_melenk_2011a,galkowski_spence_2023a} for the case without regularization.

\begin{remark}
Given any $\TT_\coarse\in\T$, the unique Galerkin approximation $\phi_\coarse \in \PP^p(\TT_\coarse)$ of $\phi$
obtained by solving~\eqref{eq:dirichlet_indirect_cfie:discretization} satisfies the
C\'ea lemma
\begin{align}\label{eq:dirichlet_indirect_cfie:cea}
	\norm{\phi - \phi_\coarse}{H^{-1/2}(\Gamma)} \lesssim \min_{\psi_\coarse \in \PP^p(\TT_\coarse)} \norm{\phi - \psi_\coarse}{H^{-1/2}(\Gamma)}.
\end{align}
\end{remark}

\subsection{A posteriori error estimation}\label{sec:dirichlet_indirect_cfie:posteriori}

We employ the standard weighted-residual error estimator 
\begin{align}
\label{eq:estimator_indirect}
\eta_\coarse
:=
\norm{h_\coarse^{1/2} \nabla_\Gamma (u - S\phi_\coarse)}{L^2(\Gamma)}.
\end{align}
An analogous estimator was introduced for the Laplace single-layer operator $V_0$ instead of
$S$ in~\cite{cs95,cs96,carstensen97}; see also \cite{bbhp19} for the Helmholtz
single-layer operator $V_k$. We show that $\eta_\coarse$ is reliable and weakly efficient.

\begin{theorem}
\label{thm:estimator_indirect}
Let $M$ be a regularizing operator with~\eqref{eq:compact}, \eqref{eq:positive}, and \eqref{eq:assumption_inverse_new}.
Then, there exist constants $\const{rel},\const{eff} >0$ such that
\begin{align}
\label{eq:reliability_indirect}
\const{rel}^{-1} \|\phi-\phi_\coarse\|_{H^{-1/2}(\Gamma)} \le \eta_\coarse
\le \const{eff} \norm{h_\coarse^{1/2} (\phi - \phi_\coarse)}{L^2(\Gamma)}
\quad\text{for all }\TT_\coarse\in\T.
\end{align}
The constants $\const{rel}$ and $\const{eff}$ depend only on $d, k, \Gamma, M$, and $\TT_0$.
\end{theorem}

\begin{remark}\label{rem:reliability}
We note that reliability, i.e., the first inequality of~\eqref{eq:reliability_indirect}, does not require the discrete inf-sup stability~\eqref{eq:dirichlet_indirect_cfie:discrete_inf-sup} and is in fact true for any (in general non-unique) $\phi_\coarse\in\PP^p(\TT_\coarse)$ satisfying \eqref{eq:dirichlet_indirect_cfie:discretization}.
Moreover, the proof of reliability does not exploit~\eqref{eq:assumption_inverse_new} but rather its implication~\eqref{eq:assumption_h1}. 
Similarly, without \eqref{eq:dirichlet_indirect_cfie:discrete_inf-sup}, weak efficiency, i.e., the second inequality of~\eqref{eq:reliability_indirect}, is still satisfied if the right-hand side is replaced by $\norm{\phi-\phi_\coarse}{H^{-1/2}(\Gamma)} +\norm{h_\coarse^{1/2} (\phi-\phi_\coarse)}{L^2(\Gamma)}$.
\end{remark}

\begin{proof}[Proof of reliability]
Since $S$ is an isomorphism, we can readily estimate the discretization error by the norm of the residual, i.e.,
\begin{align*}
	\norm{\phi - \phi_\coarse}{H^{-1/2}(\Gamma)} 
	\eqsim \norm{u - S\phi_\coarse}{H^{1/2}(\Gamma)}
\end{align*}
for all $\TT_\coarse\in\T$.
While the latter term is in principle computable, it involves a double integral over the boundary $\Gamma$ and does not provide any local information about the error.  
The localization argument of \cite{faermann00,faermann02} shows that
\begin{align}\label{eq:faermann1}
	\sum_{z\in\NN_\coarse} |v|_{H^{1/2}(\omega_\coarse(z))}^2 
	\lesssim 
	\norm{v}{H^{1/2}(\Gamma)}^2
	\lesssim
	\sum_{z\in\NN_\coarse} |v|_{H^{1/2}(\omega_\coarse(z))}^2 
	+ \norm{h_\coarse^{-1/2} v}{L^2(\Gamma)}^2
\end{align}
for all $v \in H^{1/2}(\Gamma)$.
If $v$ is orthogonal to the space of piecewise constants $\PP^0(\TT_\coarse)$, a Poincar\'e-type inequality of~\cite{faermann00,faermann02} gives that 
\begin{align*}
	\norm{h_\coarse^{-1/2} v}{L^2(\Gamma)}^2 
	\lesssim \sum_{z\in\NN_\coarse} |v|_{H^{1/2}(\omega_\coarse(z))}^2.
\end{align*}
Choosing $v = u - S\phi_\coarse$, we hence see the equivalence 
\begin{align*}
	\norm{u - S\phi_\coarse}{H^{1/2}(\Gamma)}^2
	\eqsim \sum_{z\in\NN_\coarse} |u - S\phi_\coarse|_{H^{1/2}(\omega_\coarse(z))}^2.
\end{align*}

Owing to~\eqref{eq:regularity_u}, the mapping properties
of $V_k$ and $K_k$ stated in~\eqref{eq:mapping_properties} with $\sigma = 1/2$,
and~\eqref{eq:assumption_h1}, it holds that $u - S\phi_\coarse \in H^1(\Gamma)$. 
As a result, a scaling argument allows to replace the Sobolev--Slobodeckij seminorm on patches by the $H^1$-seminorm on elements.
Indeed, it holds that 
\begin{align}\label{eq:H12_H1}
	|v|_{H^{1/2}(\omega_\coarse(z))} 
	\lesssim \norm{h_\coarse^{1/2} \nabla_\Gamma v}{L^2(\omega_\coarse(z))} 
	\quad \text{for all } v \in H^1(\Gamma) \text{ and }z \in \NN_\coarse;
\end{align}
see, e.g., \cite[Proposition~5.2.2]{gantner17} for a detailed proof. 
Owing to finite overlap of the patches, we get that 
\begin{align*}
	\norm{u - S\phi_\coarse}{H^{1/2}(\Gamma)}^2
	\lesssim \sum_{T\in\TT_\coarse} \norm{h_\coarse^{1/2} \nabla_\Gamma (u - S\phi_\coarse)}{L^2(T)}^2
	= \eta_\coarse^2,
\end{align*}
which proves the first inequality in~\eqref{eq:reliability_indirect}.
\end{proof}

\begin{proof}[Proof of weak efficiency]
We proceed similarly as in~\cite[Section~3.2.1]{affkmp17}, where weak efficiency is shown for
the Laplace problem. The definition of $S$ and the inverse estimates~\eqref{eq:invest_V},
\eqref{eq:invest_K}, and \eqref{eq:assumption_inverse_new} show that
\begin{align*}
	&\norm{h_\coarse^{1/2} \nabla_\Gamma (u-S\phi_\coarse)}{L^2(\Gamma)}
	=\norm{h_\coarse^{1/2} \nabla_\Gamma [V_k(\phi-\phi_\coarse) + i(K_k+1/2) M(\phi-\phi_\coarse)]}{L^2(\Gamma)}
	\\ 
	&\lesssim \norm{\phi-\phi_\coarse}{H^{-1/2}(\Gamma)} 
		+\norm{h_\coarse^{1/2} (\phi-\phi_\coarse)}{L^2(\Gamma)} 
		+ \norm{M(\phi-\phi_\coarse)}{H^{1/2}(\Gamma)}
		+ \norm{h_\coarse^{1/2}\nabla_\Gamma M(\phi-\phi_\coarse)}{L^2(\Gamma)} 
	\\
	&\lesssim \norm{\phi-\phi_\coarse}{H^{-1/2}(\Gamma)} 
		+\norm{h_\coarse^{1/2} (\phi-\phi_\coarse)}{L^2(\Gamma)}.
\end{align*}
Using that the Galerkin projection $\phi\mapsto\phi_\coarse$ is $H^{-1/2}(\Gamma)$-stable, one can now argue as in the proof of~\cite[Corollary~3.3]{affkmp17} to see that
\begin{align}\label{eq:approx_phi}
	\norm{\phi-\phi_\coarse}{H^{-1/2}(\Gamma)}
	\lesssim \norm{h_\coarse^{1/2} (\phi-\phi_\coarse)}{L^2(\Gamma)},
\end{align}
which concludes the proof.
\end{proof}

Combining the arguments from the proof of Theorem~\ref{thm:estimator_indirect} with standard compactness arguments even allows to derive the following estimates, which are robust with respect to high frequencies up to a term that has to be resolved by a sufficiently fine mesh size.
We note that similar arguments also apply for standard boundary integral equations provided that $k$ is not a spurious resonance.
\begin{corollary}\label{cor:asymptotic}
Let $M$ be a regularizing operator with~\eqref{eq:compact}, \eqref{eq:positive},
and \eqref{eq:assumption_inverse_new}. Then, there exist a constant $\const{rel,a}>0$
and a function
$\vartheta^k_{\rm rel,a}: \R_+ \to \R_+$ with $\lim_{h \to 0} \vartheta^k_{\rm rel,a}(h) = 0$
such that
\begin{align}\label{eq:reliability_indirect2}
\|\phi-\phi_\coarse\|_{H^{-1/2}(\Gamma)}
\le
\const{rel,a}\big(1+\vartheta^k_{\rm rel,a}(\norm{h_\coarse}{L^\infty(\Gamma)})\big)
\eta_\coarse 
\quad\text{for all }\TT_\coarse\in\T.
\end{align}
If additionally, $M \in \LL(H^{-1/2}(\Gamma),H^1(\Gamma))$,
then there also exist a constant $\const{eff,a}>0$ and a function
$\vartheta^k_{\rm eff,a}: \R_+ \to \R_+$ with $\lim_{h \to 0} \vartheta^k_{\rm eff,a}(h) = 0$
such that 
\begin{align}\label{eq:efficiency_indirect2}
\eta_\coarse
\le
\const{eff,a}\big(1+\vartheta^k_{\rm eff,a}(\norm{h_\coarse}{L^\infty(\Gamma)})\big)
\norm{h_\coarse^{1/2}(\phi-\phi_\coarse)}{L^2(\Gamma)}
\quad\text{for all }\TT_\coarse\in\T.
\end{align}
The constant $\const{rel,a}$ depends only on $d, \Gamma$, and $\TT_0$, while
$\const{eff,a}$ depends additionally on $p$. 
In particular, both $\const{rel,a}$ and $\const{eff,a}$ are independent of $k$ and $M$.
\end{corollary}

\begin{proof}[Proof of reliability]
In this proof, hidden constants are independant of $k$.
As already pointed out above, the operator
\begin{equation*}
C_k := S-V_0^+: H^{-1/2}(\Gamma) \to H^{1/2}(\Gamma)
\end{equation*}
is compact.
Since $S^{-1}: H^{1/2}(\Gamma) \to H^{-1/2}(\Gamma)$ is a continuous operator, the dual operator
$(C_k\circ S^{-1})': H^{-1/2}(\Gamma) \to H^{-1/2}(\Gamma)$ is compact
as well, and standard arguments ensure that $\lim_{h \to 0} \vartheta^k_{\rm rel,a}(h) = 0$ for
\begin{equation*}
\vartheta^k_{\rm rel,a}(h)
:=
\max_{\substack{\TT_\fine \in \T \\ h_\fine \leq h}}
\max_{\substack{\psi \in H^{-1/2}(\Gamma) \\ \norm{\psi}{H^{-1/2}(\Gamma)} = 1}}
\min_{\xi_\fine \in \PP^p(\TT_\fine)}
\norm{(C_k \circ S^{-1})' \psi-\xi_\fine}{H^{-1/2}(\Gamma)}.
\end{equation*}

Ellipticity of $V_0^+$ and the triangle inequality show that
\begin{equation*}
\norm{\phi-\phi_\coarse}{H^{-1/2}(\Gamma)}
\eqsim
\norm{V_0^+(\phi-\phi_\coarse)}{H^{1/2}(\Gamma)}
\le
\norm{S(\phi-\phi_\coarse)}{H^{1/2}(\Gamma)}
+
\norm{C_k(\phi-\phi_\coarse)}{H^{1/2}(\Gamma)}.
\end{equation*}
We have already established in the proof of Theorem~\ref{thm:estimator_indirect} that
\begin{equation*}
\norm{S(\phi-\phi_\coarse)}{H^{1/2}(\Gamma)}
=
\norm{u-S\phi_\coarse}{H^{1/2}(\Gamma)}
\lesssim
\eta_\coarse,
\end{equation*}
and we therefore focus on the second term.
There exists $\psi^\star \in H^{-1/2}(\Gamma)$ with $\norm{\psi^\star}{H^{-1/2}(\Gamma)} = 1$
such that
\begin{equation*}
\norm{C_k(\phi-\phi_\coarse)}{H^{1/2}(\Gamma)}
=
\dual{C_k(\phi-\phi_\coarse)}{\psi^\star}_{L^2(\Gamma)}.
\end{equation*}
With $\xi^\star := (C_k \circ S^{-1})' \psi^\star \in H^{-1/2}(\Gamma)$,
Galerkin orthogonality gives that
\begin{equation*}
\norm{C_k(\phi-\phi_\coarse)}{H^{1/2}(\Gamma)}
=
\dual{S(\phi-\phi_\coarse)}{\xi^\star}_{L^2(\Gamma)}
=
\dual{S(\phi-\phi_\coarse)}{\xi^\star-\xi_\coarse}_{L^2(\Gamma)}
\end{equation*}
for all $\xi_\coarse \in \PP^p(\TT_\coarse)$.
It follows that
\begin{equation*}
\norm{C_k(\phi-\phi_\coarse)}{H^{1/2}(\Gamma)}
\leq
\norm{S(\phi-\phi_\coarse)}{H^{1/2}(\Gamma)}
\min_{\xi_\coarse \in \PP^p(\TT_\coarse)}\norm{\xi^\star-\xi_\coarse}{H^{-1/2}(\Gamma)}
\lesssim
\vartheta^k_{\rm rel,a}(\norm{h_\coarse}{L^\infty(\Gamma)}) \eta_\coarse,
\end{equation*}
from which we infer the desired result.
\end{proof}

\begin{proof}[Proof of weak efficiency]
With $e_\coarse := \phi-\phi_\coarse$, we have that
\begin{equation*}
\eta_\coarse
=
\norm{h_\coarse^{1/2}\nabla_{\Gamma}(u-S\phi_\coarse)}{L^2(\Gamma)}
=
\norm{h_\coarse^{1/2}\nabla_{\Gamma}Se_\coarse}{L^2(\Gamma)},
\end{equation*}
and it follows from the definition of $S$ that
\begin{equation*}
\eta_\coarse
\leq
\norm{h_\coarse^{1/2}\nabla_{\Gamma}V_0 e_\coarse}{L^2(\Gamma)}
+
\norm{h_\coarse^{1/2}\nabla_{\Gamma}(V_k-V_0) e_\coarse}{L^2(\Gamma)}
+
\norm{h_\coarse^{1/2}\nabla_{\Gamma}(K_k+1/2) M e_\coarse}{L^2(\Gamma)},
\end{equation*}
and we now treat each term independently. First, we use the inverse estimate~\eqref{eq:invest_V}
for $k=0$, which gives that
\begin{equation*}
\|h_\coarse^{1/2} \nabla_{\Gamma} V_0 e_\coarse\|_{L^2(\Gamma)}
\lesssim
\|e_\coarse\|_{H^{-1/2}(\Gamma)} + \|h_\coarse^{1/2} e_\coarse\|_{L^2(\Gamma)},
\end{equation*}
where the hidden constant is independent of $k$.
We next employ the fact that $V_k-V_0$ continuously maps $H^{-1/2}(\Gamma)$
into $H^1(\Gamma)$, see, e.g., \cite[Proof of Theorem~2]{bbhp19}, so that
\begin{equation*}
\norm{h_\coarse^{1/2}\nabla_{\Gamma}(V_k-V_0) e_\coarse}{L^2(\Gamma)}
\leq
\norm{h_\coarse^{1/2}}{L^\infty(\Gamma)}
\norm{\nabla_{\Gamma}(V_k-V_0) e_\coarse}{L^2}(\Gamma)
\lesssim
\norm{h_\coarse^{1/2}}{L^\infty(\Gamma)} \norm{e_\coarse}{H^{-1/2}(\Gamma)}.
\end{equation*}
The inverse estimate~\eqref{eq:invest_K} and the assumption
$M \in\LL(H^{-1/2}(\Gamma),H^1(\Gamma))$ give that
\begin{align*}
\norm{h_\coarse^{1/2}\nabla_{\Gamma}(K_k+1/2) M e_\coarse}{L^2(\Gamma)}
\lesssim
\norm{h_\coarse^{1/2}\nabla_{\Gamma}M e_\coarse}{L^2(\Gamma)}
\lesssim
\norm{h_\coarse^{1/2}}{L^\infty(\Gamma)}\norm{e_\coarse}{H^{-1/2}(\Gamma)}.
\end{align*}

Finally, we use again inequality~\eqref{eq:approx_phi}. Here, a standard argument shows
that the employed stability constant of the Galerkin projection $\phi\mapsto\phi_\coarse$
as mapping from $H^{-1/2}(\Gamma)$ onto itself is also of the form 
$\const{gal,a} (1 + \vartheta^k_{\rm gal,a}(\norm{h_\coarse}{L^\infty(\Gamma)}))$
for some $k$-independent $\const{gal,a}>0$ and some $\vartheta^k_{\rm gal,a}:\R_+\to\R_+$
with $\lim_{h\to 0} \vartheta^k_{\rm gal,a}(h) = 0$. This concludes the proof.
\end{proof}

\subsection{Adaptive algorithm}\label{sec:dirichlet_indirect_cfie:algorithm}

We abbreviate the weighted-residual error indicators by
\begin{align}\label{eq:dirichlet_indirect_cfie:indicators}
	\eta_\coarse(T)^2 := \norm{h_\coarse^{1/2} \nabla_\Gamma (u - S\phi_\coarse)}{L^2(T)}^2
	\quad \text{for all } T\in\TT_\coarse.
\end{align}
Moreover, we set
\begin{align*}
	\eta_\coarse(\UU_\coarse)^2 := \sum_{T \in \UU_\coarse} \eta_\coarse(T)^2
	\quad \text{for all } \UU_\coarse \subseteq \TT_\coarse.
\end{align*}
Note that $\eta_\coarse = \eta_\coarse(\TT_\coarse)$. 
We can use the weighted-residual error estimator within a standard adaptive algorithm to steer the local mesh refinement of some given sufficiently fine conforming initial triangulation $\TT_0$ such that the discrete inf-sup stability~\eqref{eq:dirichlet_indirect_cfie:discrete_inf-sup} is satisfied.  

\begin{algorithm}\label{alg:dirichlet_indirect_cfie}
\textbf{Input:} D\"orfler parameter $0 < \theta \le 1$. 
\\
\textbf{Loop:} For each $\ell=0,1,2,\dots$, iterate the following steps:
\begin{enumerate}[\rm(i)]
\item Compute the Galerkin approximation $\phi_\ell \in \PP^p(\TT_\ell)$ by solving~\eqref{eq:dirichlet_indirect_cfie:discretization}.
\item Compute the refinement indicators $\eta_\ell({T})$ from~\eqref{eq:dirichlet_indirect_cfie:indicators} for all elements ${T}\in\TT_\ell$.
\item Determine a minimal\footnote{More precisely, we choose $\MM_\ell\subseteq\TT_\ell$ with $\theta\,\eta_\ell^2 \le \eta_\ell(\UU_\ell)^2$ such that $\#\MM_\ell \le \#\UU_\ell$ for all $\UU_\ell \subseteq \TT_\ell$ with $\theta\,\eta_\ell^2 \le \eta_\ell(\UU_\ell)^2$. Note that this choice might not be unique.} set of marked elements $\MM_\ell\subseteq\TT_\ell$ 
satisfying the D\"orfler marking
\begin{align*}
	\theta\,\eta_\ell^2 \le \eta_\ell(\MM_\ell)^2.
\end{align*}%
\item Generate the refined triangulation $\TT_{\ell+1}:=\refine(\TT_\ell,\MM_\ell)$. %
\end{enumerate}
\textbf{Output:} Refined triangulations $\TT_\ell$, corresponding Galerkin approximations $\phi_\ell \in \PP^p(\TT_\ell)$, and 
weighted-residual error estimators $\eta_\ell$ for all $\ell \in \N_0$.
\end{algorithm}

\begin{remark}\label{rem:vienna_cheat}
In \cite{bbhp19}, which investigates adaptive BEM for the standard integral equation~\eqref{eq:bie_first_kind}, the assumption that $\TT_0$ is sufficiently fine is circumvented by an alternative adaptive algorithm: 
If the discrete system on level $\ell$ is not solvable, then one uniform refinement is performed. 
Moreover, in each step an extended D\"orfler marking is used, which additionally marks the largest elements in the current mesh. 
Hence, that algorithm eventually guarantees the inf-sup condition~\eqref{eq:dirichlet_indirect_cfie:discrete_inf-sup} for all $\TT_\coarse\in\refine(\TT_L)$ for a sufficiently large $L\in\N_0$. 
The analysis of the present manuscript also covers the analogous version of that algorithm. 
However, we stress that it is unclear in practice when a discrete system is truly non-solvable and all convergence results only hold starting from level $L$. 
\end{remark}

\subsection{Optimal convergence} 
We set
\begin{align*}
 \T(N):=\set{\TT_\coarse\in\T}{\#\TT_\coarse-\#\TT_0\le N}
 \quad\text{for all }N\in\N_0
\end{align*}
and, for all $s>0$,
\begin{align}\label{eq:linconv_dirichlet_indirect_cfie}
\norm{\phi}{\mathbb{A}_s^{\rm est}}
 := \sup_{N\in\N_0}\min_{\TT_\coarse\in\T(N)}(N+1)^s\,\eta_\coarse\in[0,\infty].
\end{align}
We say that the solution $\phi\in H^{-1/2}(\Gamma)$ lies in the \textit{approximation class $s$ with respect to the estimator} if
\begin{align}\label{eq:optconv_dirichlet_indirect_cfie}
\norm{\phi}{\mathbb{A}_s^{\rm est}} <\infty.
\end{align}
By definition, $\norm{\phi}{\mathbb{A}^{\rm est}_s}<\infty$  implies that the error estimator $\eta_\coarse$
on the optimal triangulations $\TT_\coarse$ decays at least with rate $\OO\big((\#\TT_\coarse)^{-s}\big)$. 
The following theorem states that each possible rate $s>0$ is indeed be realized by Algorithm~\ref{alg:dirichlet_indirect_cfie}. 
The proof is  given in  Section~\ref{sec:axioms_dirichlet_indirect_cfie}, where we will verify the so-called \textit{axioms of adaptivity} from \cite{cfpp14} for the error estimator, which automatically imply the desired result. 


\begin{theorem}\label{thm:dirichlet_indirect_cfie}
Suppose that 
$M$ satisfies~\eqref{eq:compact},~\eqref{eq:positive} and~\eqref{eq:assumption_inverse_new}, 
and that the discrete inf-sup stability~\eqref{eq:dirichlet_indirect_cfie:discrete_inf-sup} is satisfied.
Let $(\TT_\ell)_{\ell\in\N_0}$ be the sequence of triangulations generated by Algorithm~\ref{alg:dirichlet_indirect_cfie}.
Then, for arbitrary $0<\theta\le1$, the estimator converges linearly, i.e., there exist constants $0<\rho_{\rm lin}<1$ and $\const{lin}\ge1$ such that
\begin{align*}
\eta_{\ell+j}^2\le \const{lin}\rho_{\rm lin}^j\eta_\ell^2\quad\text{for all }j,\ell\in\N_0.
\end{align*}
Moreover, there exists a constant $0<\theta_{\rm opt}\le1$ such that for all $0<\theta<\theta_{\rm opt}$, the estimator converges at optimal rate, i.e., for all $s>0$, there exist constants $c_{\rm opt},\const{opt}>0$ such that
\begin{align*}
 c_{\rm opt}\norm{\phi}{\mathbb{A}_s^{\rm est}}
 \le \sup_{\ell\in\N_0}{(\# \TT_\ell-\#\TT_0+1)^{s}}\,{\eta_\ell}
 \le \const{opt}\norm{\phi}{\mathbb{A}_s^{\rm est}}.
\end{align*}
The constant $\theta_{\rm opt}$ depends only on 
$d, k, \Gamma, M, \TT_0$, and $p$, while $\const{lin},\rho_{\rm lin}$ depend additionally on $\theta$, and $\const{opt}$ depends furthermore on $s>0$.
The constant $c_{\rm opt}$ depends only on $d$, $\#\TT_0$,  $s$, and if there exists $\ell_0$ with $\eta_{\ell_0}=0$, then also on $\ell_0$ and $\eta_0$.
\end{theorem}

\begin{remark}
We can slightly lighten the assumptions on the regularizing operator $M$
in Theorem~\ref{thm:dirichlet_indirect_cfie}. 
It suffices to suppose \eqref{eq:assumption_inverse_new} for $\psi \in \PP^p(\TT_\coarse)$ instead of $\psi \in L^2(\Gamma)$. 
With the discrete inverse estimate~\eqref{eq1:discrete_invest}, \eqref{eq:assumption_inverse_new} can thus be replaced by
\begin{align}
\label{eq:assumption_inverse}
	\norm{h_\coarse^{1/2} \nabla_\Gamma M\psi_\coarse}{L^2(\Gamma)}
	\lesssim \norm{\psi_\coarse}{H^{-1/2}(\Gamma)}
	\quad\text{for all }\psi_\coarse \in \PP^p(\TT_\coarse) \text{ and all }\TT_\coarse \in \T.
\end{align}
We additionally note that in fact,~\eqref{eq:assumption_inverse} is only required
for the axioms \eqref{item:stability} and \eqref{item:reduction} below. The other
axioms (namely~\eqref{item:orthogonality} and~\eqref{item:discrete_reliability})
only require~\eqref{eq:compact} and~\eqref{eq:positive}.
\end{remark}

\subsection{Proof of Theorem~\ref{thm:dirichlet_indirect_cfie}} \label{sec:axioms_dirichlet_indirect_cfie}
As already announced, we only need to verify the axioms of adaptivity from~\cite{cfpp14}. 
Then, the abstract result \cite[Theorem 4.1]{cfpp14} immediately implies linear convergence~\eqref{eq:linconv_dirichlet_indirect_cfie} at optimal rate~\eqref{eq:optconv_dirichlet_indirect_cfie}. 
We mention that the required axioms for the mesh refinement have already been stated in \eqref{item:child}--\eqref{item:closure}. 
So we only need to show that there exist uniform constants $\const{stab},\const{red},\const{drel},\const{qo}>0$ and $0<\rho_{\rm red}<1$ such that the following axioms of adaptivity for the estimator hold for all $\TT_\coarse\in \T$, $\TT_\fine \in \refine(\TT_\coarse)$ with corresponding patch of all refined elements $\RR_{\coarse,\fine} := \set{T\in\TT_\coarse}{\overline T \cap \overline{\bigcup(\TT_\coarse \setminus \TT_\fine)} \neq \emptyset}$, the sequence $(\TT_\ell)_{\ell\in\N_0}$ of Algorithm~\ref{alg:dirichlet_indirect_cfie}, and all $\ell,n\in\N_0$:

\begin{enumerate}[(1)]
\renewcommand{\theenumi}{A\arabic{enumi}}
\item Stability on non-refined elements: \label{item:stability}
	$|\eta_\fine(\TT_\fine \cap \TT_\coarse) - \eta_\coarse(\TT_\coarse \cap \TT_\fine)| 
	\le \const{stab} \norm{\phi_\fine - \phi_\coarse}{H^{-1/2}(\Gamma)}$;
\item Reduction on refined elements: \label{item:reduction}
	$\eta_\fine(\TT_\fine \setminus \TT_\coarse)
	\le \rho_{\rm red} \eta_\coarse(\TT_\coarse \setminus \TT_\fine)  
		+ \const{red} \norm{\phi_\fine - \phi_\coarse}{H^{-1/2}(\Gamma)}$;
\item General quasi-orthogonality: \label{item:orthogonality}
	$\sum_{j=\ell}^{\ell + n} \norm{\phi_{j+1} - \phi_j}{H^{-1/2}(\Gamma)}^2 
	\le \const{qo}^2 \eta_\ell^2$; 
\item Discrete reliability: \label{item:discrete_reliability}
	$\norm{\phi_\fine - \phi_\coarse}{H^{-1/2}(\Gamma)}
	\le \const{drel} \eta_\coarse(\RR_{\coarse,\fine}).$
\end{enumerate}

To show the axioms of adaptivity, we can proceed similarly as in~\cite{bbhp19}, where standard boundary integral equations for the Helmholtz problem are treated.
In addition, we crucially exploit the inverse estimate~\eqref{eq:assumption_inverse_new} or rather its implication~\eqref{eq:discrete_invest}.
Moreover, instead of building on \eqref{item:orthogonality} derived in~\cite{bbhp19}, we obtain \eqref{item:orthogonality} from~\cite{bbhp19}.
In contrast to the argument of \cite{bbhp19}, this results in constants that do not depend on the sequence of Galerkin approximations $(\phi_\ell)_{\ell\in\N_0}$; see also Remark~\ref{rem:dirichlet_indirect_cfie:cea}.

\begin{proof}[Proof of stability on non-refined elements~\eqref{item:stability}]
Let $\TT_\coarse\in\T$ and $\TT_\fine\in\refine(\TT_\coarse)$. 
The inverse triangle inequality and the fact that $h_\coarse = h_\fine$ on $\omega := \bigcup (\TT_\coarse \cap \TT_\fine)$ show that 
\begin{align*}
	\big|\norm{h_\fine^{1/2} \nabla_\Gamma (u - S\phi_\fine)}{L^2(\omega)} - \norm{h_\coarse^{1/2} \nabla_\Gamma (u - S\phi_\coarse)}{L^2(\omega)}\big|
	\le \norm{h_\fine^{1/2} \nabla_\Gamma S(\phi_\fine - \phi_\coarse)}{L^2(\Gamma)}.
\end{align*}
The inverse estimates~\eqref{eq:invest_V}--\eqref{eq:invest_K}, \eqref{eq:discrete_invest}, and \eqref{eq:assumption_inverse} on the mesh $\TT_\fine$ imply that 
\begin{align*}
	\norm{h_\fine^{1/2} \nabla_\Gamma S(\phi_\fine - \phi_\coarse)}{L^2(\Gamma)} 
	&\lesssim \norm{\phi_\fine - \phi_\coarse}{H^{-1/2}(\Gamma)} + \norm{M(\phi_\fine - \phi_\coarse)}{H^{1/2}(\Gamma)} 
	\\&\quad
		+ \norm{h_\coarse^{1/2} \nabla_\Gamma M(\phi_\fine-\phi_\coarse)}{L^2(\Gamma)}
	\\
	&\lesssim \norm{\phi_\fine - \phi_\coarse}{H^{-1/2}(\Gamma)}.
\end{align*}
This concludes the proof.
\end{proof}

\begin{proof}[Proof of reduction on refined elements~{\rm \eqref{item:reduction}}]
Let $\TT_\coarse\in\T$ and $\TT_\fine\in\refine(\TT_\coarse)$. 
The triangle inequality and the fact that $h_\fine^{1/2} \le \rho h_\coarse^{1/2}$ on $\omega := \bigcup (\TT_\fine \setminus \TT_\coarse) = \bigcup (\TT_\coarse \setminus \TT_\fine)$ for some $0<\rho<1$ show that 
\begin{align*}
	\norm{h_\fine^{1/2} \nabla_\Gamma (u - S\phi_\fine)}{L^2(\omega)} 
	\le \rho \norm{h_\coarse^{1/2} \nabla_\Gamma (u - S\phi_\coarse)}{L^2(\omega)}
	+
	 \norm{h_\fine^{1/2} \nabla_\Gamma S(\phi_\fine - \phi_\coarse)}{L^2(\Gamma)}.
\end{align*}
The second term has already been estimated in the proof of \eqref{item:stability}. 
\end{proof}

\begin{proof}[Proof of general quasi-orthogonality~{\rm \eqref{item:orthogonality}}]
By discrete inf-sup stability~\eqref{eq:dirichlet_indirect_cfie:discrete_inf-sup}, \cite{feischl22} is applicable and directly implies this axiom with $\const{qo}=\const{qo}(n)=o(n)$ depending sublinearly on the integer~$n$.
The post-processing argument~\cite[Remark~20]{feischl22} or \cite[Proposition~4.11]{cfpp14} then yield the axiom with a uniform constant~$\const{qo}\simeq 1$. 
\end{proof}

\begin{remark}\label{rem:dirichlet_indirect_cfie:cea}
Since $S$ is symmetric and elliptic up to some compact perturbation, general quasi-orthogonality~{\rm \eqref{item:orthogonality}} alternatively follows as in \cite[Section~3--4]{ffp14} and \cite[Proposition~4.11]{cfpp14}; see also the abstract framework of~\cite[Section~3.3]{bhp17}. 
The drawback of this approach is that the resulting $\const{qo}$ depends on the sequence of Galerkin approximations $(\phi_\ell)_{\ell\in\N_0}$.  

That said, the underlying analysis allows to verify Corollary~\ref{cor:asymptotic} along the mesh sequence $(\TT_\ell)_{\ell\in\N_0}$.
More precisely, reliability~\eqref{eq:reliability_indirect2} and weak efficiency~\eqref{eq:efficiency_indirect2} hold for $\vartheta^k_{\rm rel,a}(\norm{h_\ell}{L^\infty(\Gamma)})$ and $\vartheta^k_{\rm eff,a}(\norm{h_\ell}{L^\infty(\Gamma)})$ replaced by some 
$\vartheta^{k,\ell}_{\rm rel,a}>0$ and $\vartheta^{k,\ell}_{\rm eff,a}>0$ which both converge to $0$ as $\ell\to\infty$. 
To see this, one can repeat the arguments of Corollary~\ref{cor:asymptotic} and exploit the crucial result from \cite{bhp17} that the normed error sequence $\big(\tfrac{\phi - \phi_\ell}{\norm{\phi-\phi_\ell}{H^{-1/2}(\Gamma)}}\big)_{\ell\in\N_0}$ converges weakly to $0$ as well as the fact that compact operators turn weak convergence into strong convergence; see \cite[Theorem~20]{bhp17} for a similar argument used to prove an asymptotically constant-free C\'ea lemma. 
The proof of weak efficiency now additionally employs that $V_k-V_0:H^{-1/2}(\Gamma)\to H^1(\Gamma)$ is even compact and further requires that $M:H^{-1/2}(\Gamma)\to H^1(\Gamma)$ is compact.
\end{remark}

\begin{proof}[Proof of discrete reliability~\eqref{item:discrete_reliability}] 
The proof follows as for the weakly-singular integral equation of the Laplace operator~\cite{fkmp13,gantumur13}  or the standard weakly-singular integral equation for the Helmholtz operator~\cite{bbhp19}, 
and is only given for the sake of completeness. 
Discrete inf-sup stability~\eqref{eq:dirichlet_indirect_cfie:discrete_inf-sup} and the definition of $\phi_\fine$ as well $\phi_\coarse$ (see~\eqref{eq:dirichlet_indirect_cfie:discretization}) show for all $\psi_\coarse\in \PP^p(\TT_\coarse)$  that 
\begin{align*}
	\norm{\phi_\fine - \phi_\coarse}{H^{-1/2}(\Gamma)}
	\lesssim \sup_{\psi_\fine \in \PP^p(\TT_\fine)}  \frac{|\dual{S(\phi_\fine - \phi_\coarse)}{\psi_\fine}_{L^2(\Gamma)}|}{\norm{\psi_\fine}{H^{-1/2}(\Gamma)}}
	= \sup_{\psi_\fine \in \PP^p(\TT_\fine)}  \frac{|\dual{S(\phi - \phi_\coarse)}{\psi_\fine - \psi_\coarse}_{L^2(\Gamma)}|}{\norm{\psi_\fine}{H^{-1/2}(\Gamma)}}.
\end{align*}
We choose
\begin{align}\label{eq:psi_drel}
	\psi_\coarse 
	:= \begin{cases} 0 &\text{ on } \bigcup (\TT_\coarse\setminus \TT_\fine),
	\\
	\psi_\fine &\text{ on } \bigcup (\TT_\coarse\cap \TT_\fine).
	\end{cases}
\end{align}
In particular, the support of $\psi_\fine - \psi_\coarse$ is contained in $\omega := \bigcup (\TT_\coarse\setminus \TT_\fine)$. 
With the hat functions $\varphi_{\coarse,z}\in \mathcal{S}^1(\TT_\coarse)$
defined for all $z \in \NN_\coarse$ via $\varphi_{\coarse,z}(z')=\delta_{z,z'}$ for all $z'\in\NN_\coarse$, the continuous cut-off function $\widetilde \chi_\omega :=\sum_{z\in\NN_\coarse \cap \omega} \varphi_{\coarse,z}$ equals $1$ on the set $\omega$.
This shows that 
\begin{align*}
	\dual{S(\phi - \phi_\coarse)}{\psi_\fine - \psi_\coarse}_{L^2(\Gamma)}
	= \dual{\widetilde\chi_\omega S(\phi - \phi_\coarse)}{\psi_\fine - \psi_\coarse}_{L^2(\Gamma)}.
\end{align*}
We bound $\dual{\widetilde\chi_\omega S(\phi - \phi_\coarse)}{\psi_\fine}_{L^2(\Gamma)}$ and $\dual{\widetilde\chi_\omega S(\phi - \phi_\coarse)}{\psi_\coarse}_{L^2(\Gamma)}$ separately.  
The Cauchy--Schwarz inequality, the bounds $0 \le \widetilde\chi_\omega \le 1$ with ${\rm supp} (\widetilde\chi_\omega) \subseteq \bigcup\RR_{\coarse,\fine}$, the fact that $h_\coarse = h_\fine$ and $\psi_\coarse = \psi_\fine$ on $\bigcup(\TT_\coarse\cap\TT_\fine)$, the inverse estimate~\eqref{eq:discrete_invest}, and an element-wise Poincar\'e inequality (see, e.g., \cite[Lemma~5.2.1 and 5.3.3]{gantner17})
show for the latter term that
\begin{align*}
	|\dual{\widetilde\chi_\omega S(\phi - \phi_\coarse)}{\psi_\coarse}_{L^2(\Gamma)}|
	&\le \norm{h_\coarse^{-1/2}\widetilde\chi_\omega S(\phi - \phi_\coarse)}{L^2(\Gamma)} \norm{h_\coarse^{1/2}\psi_\coarse}{L^2(\bigcup (\TT_\coarse \cap \TT_\fine))}
	\\
	&\le \norm{h_\coarse^{-1/2} S(\phi - \phi_\coarse)}{L^2(\bigcup \RR_{\coarse,\fine})} \norm{h_\fine^{1/2}\psi_\fine}{L^2(\Gamma)}
	\\
	&\lesssim \norm{h_\coarse^{1/2} \nabla_\Gamma S(\phi - \phi_\coarse)}{L^2(\bigcup \RR_{\coarse,\fine})} \norm{\psi_\fine}{H^{-1/2}(\Gamma)}.
\end{align*}
For the other term, we have that 
\begin{align*}
	|\dual{\widetilde\chi_\omega S(\phi - \phi_\coarse)}{\psi_\fine}_{L^2(\Gamma)}|
	\le \norm{\widetilde\chi_\omega S(\phi - \phi_\coarse)}{H^{1/2}(\Gamma)} \norm{\psi_\fine}{H^{-1/2}(\Gamma)}.
\end{align*}
With the Faermann estimate~\eqref{eq:faermann1}, we see that 
\begin{align*}
	\norm{\widetilde\chi_\omega S(\phi - \phi_\coarse)}{H^{1/2}(\Gamma)}^2
	\lesssim \sum_{z\in\NN_\coarse} |\widetilde\chi_\omega  S(\phi - \phi_\coarse)|_{H^{1/2}(\omega_\coarse(z))}^2  + \norm{h_\coarse^{-1/2} \widetilde\chi_\omega  S(\phi - \phi_\coarse)}{L^2(\Gamma)}^2.
\end{align*}
We have already estimated the $L^2$-term before. 
The local bound~\eqref{eq:H12_H1}, the product rule, and the properties of $\widetilde\chi_\omega$ show for the other term that 
\begin{align*}
	\sum_{z\in\NN_\coarse} |\widetilde\chi_\omega  S(\phi - \phi_\coarse)|_{H^{1/2}(\omega_\coarse(z))}^2
	&\lesssim \sum_{z\in\NN_\coarse} \norm{h_\coarse^{1/2} \nabla_\Gamma (\widetilde\chi_\omega  S(\phi - \phi_\coarse))}{L^2(\omega_\coarse(z))}^2
	\\
	&\lesssim \norm{h_\coarse^{1/2} \nabla_\Gamma \widetilde\chi_\omega  S(\phi - \phi_\coarse)}{L^2(\Gamma)}^2
	+\norm{h_\coarse^{1/2} \nabla_\Gamma  S(\phi - \phi_\coarse)}{L^2(\bigcup\RR_{\coarse,\fine})}^2
	\\
	&\lesssim \norm{h_\coarse^{-1/2}  S(\phi - \phi_\coarse)}{L^2(\bigcup\RR_{\coarse,\fine})}^2
	+\norm{h_\coarse^{1/2} \nabla_\Gamma  S(\phi - \phi_\coarse)}{L^2(\bigcup\RR_{\coarse,\fine})}^2.
\end{align*}
Again, the first term has been already estimated before. 
Overall, we conclude that 
\begin{align*}
	|\dual{S(\phi - \phi_\coarse)}{\psi_\fine - \psi_\coarse}_{L^2(\Gamma)}|
	\le \eta_\coarse(\RR_{\coarse,\fine}) \norm{\psi_\fine}{H^{-1/2}(\Gamma)}
\end{align*}
which completes the proof.
\end{proof}

\section{Mixed form of indirect integral equation} \label{sec:dirichlet_indirect_mixed}

As discussed in Section~\ref{section_M}, the regularizing operator $M$ can be realized with integral operators,
or with the inverse of a differential operator. The former might be cumbersome since
the Galerkin matrices associated to the composition of integral operators are typically challenging
to assemble. In this section, we consider instead the latter option.
Specifically, we consider the choice $M := (\scal -\Delta_\Gamma)^{-1}: H^{-1/2}(\Gamma) \to H^1(\Gamma)$, presented in Section~\ref{section_M}.

While $M$ satisfies all required assumptions of Section~\ref{sec:dirichlet_indirect_cfie},
the discretization of~\eqref{eq:dirichlet_indirect_cfie} is not
feasible, as $M$ cannot be computed exactly (even up to quadrature errors).
Instead, we introduce similarly as in~\cite{bh05} the extra variable $f := M\phi$ and rewrite \eqref{eq:dirichlet_indirect_cfie}
as the following mixed system: Find $(\phi,f) \in H^{-1/2}(\Gamma) \times H^1(\Gamma)$ such that
\begin{align}\label{eq:dirichlet_indirect_mixed}
\begin{split}
	V_k \phi + i (K_k+1/2) f &= u,
	\\
	- \phi + \scal f  - \Delta_\Gamma f &= 0,
\end{split}
\end{align}
where the equations hold in $H^{1/2}(\Gamma) \times H^{-1}(\Gamma)$. 
Since $V_k\in \LL(H^{-1/2}(\Gamma),H^{1/2}(\Gamma))$ and $\scal-\Delta_\Gamma\in \LL(H^1(\Gamma),H^{-1}(\Gamma))$ are symmetric and elliptic at least up to some compact perturbation and $K_k\in \LL(H^1(\Gamma),H^1(\Gamma))$, the operator $R = (R_1,R_2)\in \LL( H^{-1/2}(\Gamma) \times H^1(\Gamma),H^{1/2}(\Gamma) \times H^{-1}(\Gamma))$ induced by the left-hand side of~\eqref{eq:dirichlet_indirect_mixed} is elliptic up to some compact perturbation. 
As~\eqref{eq:dirichlet_indirect_cfie} is uniquely solvable, so is~\eqref{eq:dirichlet_indirect_mixed}. Hence, $R$ is injective and thus even bijective by the Fredholm alternative.
We conclude that problem~\eqref{eq:dirichlet_indirect_mixed} is well-posed.

\begin{remark}
With minor modifications, the analysis of the whole Section~\ref{sec:dirichlet_indirect_mixed}
holds equally true for the (inverse of the) piecewise Laplace--Beltrami operator
of~\eqref{eq:pw_laplace_beltrami}. For brevity though, we do not provide the details.
\end{remark}

\subsection{Discretization}\label{sec:dirichlet_indirect_mixed:discretization}

For any $\TT_\coarse \in \T$, we consider the discrete subspace
$\PP^p(\TT_\coarse) \times \mathcal{S}^{p+2}(\TT_\coarse) \subset H^{-1/2}(\Gamma) \times H^1(\Gamma)$.
This leads to the discrete formulation: Find a Galerkin approximation
$(\phi_\coarse,f_\coarse) \in \PP^p(\TT_\coarse)\times \mathcal{S}^{p+2}(\TT_\coarse)$
such that
\begin{align}\label{eq:dirichlet_indirect_mixed:discretization}
	\dual{R(\phi_\coarse,f_\coarse)}{(\psi_\coarse,g_\coarse)}_{L^2(\Gamma)} = \dual{u}{\psi_\coarse}_{L^2(\Gamma)}
	\quad \text{for all } (\psi_\coarse,g_\coarse) \in \PP^p(\TT_\coarse) \times \mathcal{S}^{p+2}(\TT_\coarse). 
\end{align}

To theoretically guarantee well-posedness of~\eqref{eq:dirichlet_indirect_mixed:discretization},
we assume that the initial triangulation $\TT_0$ is sufficiently fine such that
\begin{align}\label{eq:dirichlet_indirect_mixed:discrete_inf-sup}
	\inf_{\chi_\coarse \in \PP^p(\TT_\coarse)\atop e_\coarse\in \mathcal{S}^{p+2}(\TT_\coarse)}\sup_{\psi_\coarse \in \PP^p(\TT_\coarse)\atop g_\coarse \in \mathcal{S}^{p+2}(\TT_\coarse)}  
	\frac{|\dual{R (\chi_\coarse,e_\coarse)}{(\psi_\coarse,g_\coarse)}_{L^2(\Gamma)}|}{\norm{(\psi_\coarse,g_\coarse)}{H^{-1/2}(\Gamma)\times H^1(\Gamma)}}
	\gtrsim 1
	\quad \text{for all }\TT_\coarse\in\T,
\end{align}
where $\dual{\cdot}{\cdot}_{L^2(\Gamma)}$ denotes the extended $L^2$-scalar product on $[H^{1/2}(\Gamma) \times H^{-1}(\Gamma)] \times [H^{-1/2}(\Gamma)\times H^1(\Gamma)]$.
Owing to the fact that $R$ is symmetric and elliptic up to some compact perturbation, the existence of
a maximal mesh size ${\tt h} \gtrsim 1$ ensuring~\eqref{eq:dirichlet_indirect_mixed:discrete_inf-sup}
whenever $h_T \leq {\tt h}$ for all $T \in \TT_0$ is standard; see again, e.g., \cite[Theorem~4.2.9]{ss11} or \cite[Proposition~1]{bhp17}. 

\begin{remark}
Given any $\TT_\coarse\in\T$, the unique Galerkin approximation
$(\phi_\coarse,f_\coarse) \in \PP^p(\TT_\coarse)\times \mathcal{S}^{p+2}(\TT_\coarse)$
of $(\phi,f)$ obtained by solving~\eqref{eq:dirichlet_indirect_mixed:discretization}
satisfies the C\'ea lemma
\begin{align}\label{eq:dirichlet_indirect_mixed:cea}
	\norm{(\phi - \phi_\coarse, f - f_\coarse)}{H^{-1/2}(\Gamma) \times H^1(\Gamma)}
	\lesssim \min_{\psi_\coarse \in \PP^p(\TT_\coarse) \atop g_\coarse \in \mathcal{S}^{p+2}(\TT_\coarse)} 
		\norm{(\phi - \psi_\coarse, f - g_\coarse)}{H^{-1/2}(\Gamma)\times H^1(\Gamma)}.
\end{align}
\end{remark}

\begin{remark}
For sufficiently smooth $\phi$ and $f$, one has that
\begin{align*}
	\min_{\psi_\coarse \in \PP^p(\TT_\coarse)} \norm{\phi - \psi_\coarse}{H^{-1/2}(\Gamma)} 
	&= \OO(\norm{h_\coarse}{L^\infty(\Gamma)}^{p+3/2})
	\\
	\min_{g_\coarse \in \mathcal{S}^{p+2}(\TT_\coarse)} \norm{f - g_\coarse}{H^1(\Gamma)}
	&= \OO(\norm{h_\coarse}{L^\infty(\Gamma)}^{p+2}).
\end{align*}
The choice $p+2$ as polynomial degree to discretize $f$ thus does not spoil the convergence rate of $\norm{\phi-\phi_\coarse}{H^{-1/2}(\Gamma)}$. 
That being said, the forthcoming analysis holds analogously for $p+2$ replaced by some arbitrary fixed degree $q\in\N$. 
\end{remark}

\subsection{A posteriori error estimation}

Our error estimator now consists of two parts. For the first equation
in~\eqref{eq:dirichlet_indirect_mixed}, we introduce
\begin{subequations}
\begin{equation}\label{eq:mixed_est1}
\eta_{1,\coarse}
:=
\|h_\coarse^{1/2}\nabla_\Gamma(u-V_k\phi_\coarse-i(K_k+1/2)f_\coarse)\|_{L^2(\Gamma)},
\end{equation}
whereas
\begin{equation}\label{eq:mixed_est2}
\eta_{2,\coarse}^2
:=
\sum_{T\in\TT_\coarse} h_T^2 \norm{\phi_\coarse - \scal f_\coarse + \Delta_T f_\coarse}{L^2(T)}^2
+ h_T\norm{[\nabla_\Gamma f_\coarse \cdot \nnu]}{L^2(\partial T)}^2 
\end{equation}
controls the residual of the second equation. We then set
\begin{equation}
\eta_{\coarse}^2 := \eta_{1,\coarse}^2 + \eta_{2,\coarse}^2.
\end{equation}
\end{subequations}
Here, $\Delta_T$ denotes the local Laplace--Beltrami operator on $T$ and
$[\nabla_\Gamma(\cdot)\cdot \nnu]$ the jump of the normal derivative,
which reads on any non-empty face $F=T_-\cap T_+$,
$T_\pm \in\TT_\coarse$, as 
\begin{align*}
[\nabla_\Gamma(\cdot)\cdot \nnu]|_F
=
\big(\nabla_\Gamma(\cdot)|_{T_-}\cdot \nnu_- + \nabla_\Gamma(\cdot)|_{T_+}\cdot \nnu_+\big)
\end{align*}
with the tangential normal vectors $\nnu_\pm$ on $\partial T_\pm$ pointing outside of $T_\pm$
and tangent to $T_\pm$; cf.\ \cite[Section~3.1]{bcmmn16} for a precise definition in terms of the bi-Lipschitz parametrizations $\gamma_{T_\pm}$.
Note that the jump is actually defined up
to a sign (if the roles of $T_\pm$ are flipped), which is however not relevant here, as only the magnitude
enters the definition of the estimator.

The first part~\eqref{eq:mixed_est1} of the estimator is similar to the estimator introduced
for the non-mixed formulation in Section~\ref{sec:dirichlet_indirect_cfie:posteriori}.
It is the standard weighted-residual error estimator for integral equations of the first kind;
see again~\cite{cs95,cs96,carstensen97,bbhp19}. The second part~\eqref{eq:mixed_est2} is the
standard residual estimator for Poisson problems, here applied
on a manifold; see \cite{holst01,dd07,mmn11,bcmn13,bcmmn16}.

\begin{theorem}
\label{thm:estimator_mixed}
There exists a constant $\const{rel} > 0$ such that
\begin{equation}\label{eq:reliability_mixed}
\|(\phi,f)-(\phi_\coarse,f_\coarse)\|_{H^{-1/2}(\Gamma) \times H^1(\Gamma)}
\le \const{rel}
\eta_\coarse
\quad \text{for all }\TT_\coarse\in\T.
\end{equation}
If the boundary $\Gamma$ is polyhedral and all parametrizations $\gamma_T$, $T\in\TT_0$, are chosen to be affine, 
there further exists a constant $\const{eff} > 0$ such that
\begin{align}\label{eq:dirichlet_indirect_mixed:efficiency}
	\eta_\coarse \le \const{eff} \big( \norm{h_\coarse^{1/2} (\phi - \phi_\coarse)}{L^2(\Gamma)} + \norm{f - f_\coarse}{H^1(\Gamma)}\big). 
\end{align}
The constants $\const{rel}$ and $\const{eff}$ depend only on $d, k, \Gamma, \alpha$, and $\TT_0$.
\end{theorem}

\begin{remark}\label{rem:reliability_mixed}
Similarly as in Remark~\ref{rem:reliability}, the discrete inf-sup
stability~\eqref{eq:dirichlet_indirect_mixed:discrete_inf-sup} is not
required for reliablity~\eqref{eq:reliability_mixed}, which holds in fact true for any
(in general non-unique) $(\phi_\coarse,f_\coarse)\in\PP^p(\TT_\coarse)\times \mathcal{S}^{p+2}(\TT_\coarse)$ satisfying \eqref{eq:dirichlet_indirect_mixed:discretization}.
This is also the case for weak efficiency~\eqref{eq:dirichlet_indirect_mixed:efficiency} if the right-hand side is replaced by $\norm{\phi - \phi_\coarse}{H^{-1/2}(\Gamma)} +\norm{h_\coarse^{1/2} (\phi - \phi_\coarse)}{L^2(\Gamma)} + \norm{f-f_\coarse}{H^1(\Gamma)}$.
Moreover, the assumption of a polyhedral boundary $\Gamma$ is not essential for weak efficiency~\eqref{eq:dirichlet_indirect_mixed:efficiency}, but additional oscillation terms are present in the general case; see~\cite{dd07,bcmn13,bcmmn16}.
\end{remark}

\begin{proof}[Proof of reliatility]
Since $R$ is an isomorphism, we can readily estimate
the discretization error by the norm of the residual, i.e.,
\begin{align*}
\norm{(\phi,f) - (\phi_\coarse, f_\coarse)}{H^{-1/2}(\Gamma) \times H^1(\Gamma)} 
\eqsim \norm{(u,0) - R(\phi_\coarse,f_\coarse)}{H^{1/2}(\Gamma)\times H^{-1}(\Gamma)}.
\end{align*}
As the first component $u - R_1(\phi_\coarse,f_\coarse) = u - V_k \phi_\coarse - i (K_k+1/2) f_\coarse$ is orthogonal to piecewise constants $\PP^0(\TT_\coarse)$, its  norm can be estimated as in the proof of Theorem~\ref{thm:estimator_indirect}, i.e., 
\begin{align}\label{eq:dirichlet_indirect_mixed:reliability1}
\norm{u - R_1(\phi_\coarse,f_\coarse)}{H^{1/2}(\Gamma)}^2
\eqsim \sum_{z\in\NN_\coarse} |u - R_1(\phi_\coarse,f_\coarse)|_{H^{1/2}(\omega_\coarse(z))}^2
\lesssim \eta_{1,\coarse}^2.
\end{align}
For the second component $R_2(\phi_\coarse,f_\coarse) = -\phi_\coarse + \scal f_\coarse - \Delta_\Gamma f_\coarse$, standard arguments as for the Poisson problem show that
\begin{align} \label{eq:equivalence_weighted_res}
\norm{R_2(\phi_\coarse,f_\coarse)}{H^{-1}(\Gamma)}^2 \lesssim \eta_{2,\coarse}^2;
\end{align}
cf.\ \cite{holst01} and \cite{dd07,mmn11,bcmn13,bcmmn16}, which even take some potential geometry approximation into account.
\end{proof}

\begin{proof}[Proof of weak efficiency]
In case of a polyhedral boundary $\Gamma$ (and affine parametrizations), standard arguments and the fact that $\phi_\coarse$ is discrete also yield the converse estimate of~\eqref{eq:dirichlet_indirect_mixed:reliability1} without any oscillation term, i.e., 
$\eta_{2,\coarse}\lesssim\norm{R_2(\phi_\coarse,f_\coarse)}{H^{-1}(\Gamma)}$; 
cf.~\cite{dd07,bcmn13,bcmmn16}. 
With the C\'ea lemma~\eqref{eq:dirichlet_indirect_mixed:cea}, one then sees as in the proof of Theorem~\ref{thm:estimator_indirect} the overall weak efficiency~\eqref{eq:dirichlet_indirect_mixed:efficiency}.
\end{proof}

As in Corollary~\ref{cor:asymptotic}, standard compactness arguments even allow to derive the following estimates, which are robust with respect to high frequencies up to a term that has to be resolved by a sufficiently fine mesh.

\begin{corollary}\label{cor:asymptotic2}
There exist a constant $\const{rel,a}>0$ and a function $\vartheta^{k,\alpha}_{\rm rel,a}: \R_+ \to \R_+$ with $\lim_{h \to 0} \vartheta^{k,\alpha}_{\rm rel}(h) = 0$ such that
\begin{align}
\|(\phi-\phi_\coarse,f-f_\coarse)\|_{H^{-1/2}(\Gamma) \times H^1(\Gamma)}
\le
\const{rel,a}\big(1+\vartheta^{k,\alpha}_{\rm rel,a}(\norm{h_\coarse}{L^\infty(\Gamma)})\big) \eta_\coarse 
\quad\text{for all }\TT_\coarse\in\T.
\end{align}
If the boundary $\Gamma$ is polyhedral and all parametrizations $\gamma_T$, $T\in\TT_0$, are chosen to be affine, 
there further exist a constant $\const{eff,a}>0$ and a function $\vartheta^{k,\alpha}_{\rm eff,a}: \R_+ \to \R_+$
with $\lim_{h \to 0} \vartheta^{k,\alpha}_{\rm eff,a}(h) = 0$ such that 
\begin{align}
\eta_\coarse
\le 
\const{eff,a}\big(1+\vartheta^{k,\alpha}_{\rm eff,a}(\norm{h_\coarse}{L^\infty(\Gamma)})\big)
\big(\norm{h_\coarse^{1/2} (\phi - \phi_\coarse)}{L^2(\Gamma)} + \norm{f - f_\coarse}{H^1(\Gamma)}\big)
\quad\text{for all }\TT_\coarse\in\T.
\end{align}
The constant $\const{rel,a}$ depends only on $d, \Gamma$, and $\TT_0$, while $\const{eff,a}$ depends additionally on $p$. 
In particular, both $\const{rel,a}$ and $\const{eff,a}$ are independent of $k$ and $\alpha$. 
\end{corollary}

\begin{proof}[Proof of reliability]
The proof follows along the lines of Corollary~\ref{cor:asymptotic}, where the compact operator 
\begin{align*}
	C_{k,\alpha} := R -\left [
	\begin{array}{cc}
	V_0^+ & 0
	\\
	0 & -\Delta_\Gamma
	\end{array}
	\right ]: H^{-1/2}(\Gamma) \times H^1(\Gamma) \to H^{-1/2}(\Gamma) \times H^{-1}(\Gamma)
\end{align*}
now plays the role of $C_k$ from before. 
\end{proof}

\begin{proof}[Proof of weak efficiency]
For the first part of the estimator, we can proceed as in Corollary~\ref{cor:asymptotic} to see that
\begin{align*}
&\eta_{1,\coarse}
=
\norm{h_\coarse^{1/2} \nabla_\Gamma (V_k(\phi-\phi_\coarse) + i(K_k+1/2)(f-f_\coarse)}{L^2(\Gamma)}
\\
&\leq
\norm{h_\coarse^{1/2} \nabla_\Gamma V_0(\phi-\phi_\coarse)}{L^2(\Gamma)}
+
\norm{h_\coarse^{1/2} \nabla_\Gamma (V_k-V_0)(\phi-\phi_\coarse)}{L^2(\Gamma)}
+
\norm{h_\coarse^{1/2} \nabla_{\Gamma}(K_k+1/2)(f-f_\coarse)}{L^2(\Gamma)}
\end{align*}
with
\begin{align*}
\norm{h_\coarse^{1/2} \nabla_\Gamma V_0(\phi-\phi_\coarse)}{L^2(\Gamma)}
&\lesssim
\norm{\phi-\phi_\coarse}{H^{-1/2}(\Gamma)}
+
\norm{h_\coarse^{1/2} (\phi-\phi_\coarse)}{L^2(\Gamma)},
\\
\norm{h_\coarse^{1/2} \nabla_\Gamma (V_k-V_0)(\phi-\phi_\coarse)}{L^2(\Gamma)}
&\lesssim 
\norm{h_\coarse^{1/2}}{L^\infty(\Gamma)} \norm{\phi-\phi_\coarse}{H^{-1/2}(\Gamma)},
\\
\norm{h_\coarse^{1/2} \nabla_{\Gamma}(K_k+1/2)(f-f_\coarse)}{L^2(\Gamma)}
&\lesssim
\norm{h_\coarse^{1/2}}{L^\infty(\Gamma)} \norm{f-f_\coarse}{H^1(\Gamma)},
\end{align*}
where the hidden constant in the first ``$\lesssim$'' does not depend on $k$ and $\alpha$.
A standard argument shows that the stability constant of the Galerkin projection $(\phi,f)\mapsto(\phi_\coarse,f_\coarse)$ as mapping from $H^{-1/2}(\Gamma)\times H^1(\Gamma)$ onto itself is also of the form 
$\const{gal,a} (1 + \vartheta^{k,\alpha}_{\rm gal,a}(\norm{h_\coarse}{L^\infty(\Gamma)}))$ for some
$k$- and $\alpha$-independent $\const{gal,a}>0$ and some $\vartheta^{k,\alpha}_{\rm gal,a}:\R_+\to\R_+$ with $\lim_{h\to 0} \vartheta^{k,\alpha}_{\rm gal,a}(h) = 0$.
As in \cite[Corollary~3.3]{affkmp17}, we thus see that
\begin{align*}
\norm{\phi-\phi_\coarse}{H^{-1/2}(\Gamma)} 
\lesssim 
(1 + \vartheta^{k,\alpha}_{\rm gal,a}(\norm{h_\coarse}{L^\infty(\Gamma)})) 
\big(\norm{h_\coarse^{1/2}(\phi-\phi_\coarse)}{L^2(\Gamma)} + \norm{f - f_\coarse}{H^1(\Gamma)}\big)
\end{align*}
with hidden constant independent of $k$ and $\alpha$.

For the second part, we first fix $T \in \TT_\coarse$ and let
$r_T := \phi_\coarse-\alpha f_\coarse+\Delta_T f_\coarse$.
Using the standard bubble function $b_T$ (see e.g.~\cite{melenk_wohlmuth_2001a}),
and letting $v_T := b_T r_T$, we have that
\begin{align*}
h_T\norm{r_T}{L^2(T)}^2
&\lesssim
h_T\dual{r_T}{v_T}_{L^2(T)}
\\
&=
h_T\dual{-(\phi-\phi_\coarse)+\alpha(f-f_\coarse)-\Delta_T(f-f_\coarse)}{v_T}_{L^2(T)}
\\
&\leq
h_T \norm{\phi-\phi_\coarse}{L^2(T)}\norm{v_T}{L^2(T)}
+
(\alpha h_T) \norm{f-f_\coarse}{L^2(T)}\norm{v_T}{L^2(T)}
\\
&\quad+
\norm{\nabla_{\Gamma}(f-f_\coarse)}{L^2(T)}
h_T \norm{\nabla_{\Gamma}v_T}{L^2(T)},
\end{align*}
and an inverse estimate together with standard properties of $b_T$ show that
\begin{equation*}
h_T \norm{r_T}{L^2(T)}
\lesssim
h_T^{1/2} \norm{h_\coarse^{1/2}(\phi-\phi_\coarse)}{L^2(T)}
+
(1+\alpha h_T) \norm{f-f_\coarse}{H^1(T)}.
\end{equation*}
For $r_{\partial T} := [\nabla_\Gamma f_\coarse \cdot \nnu]$ and $\omega_\coarse(T) := \set{T'\in\TT_\coarse}{T \text{ and }T' \text{ share a facet}}$,
classically~\cite{melenk_wohlmuth_2001a}, there exists a function $v_{\partial T} \in H^1_0(\omega_\coarse(T))$
such that
\begin{equation*}
h_T^{1/2} \norm{r_{\partial T}}{L^2(\partial T)}^2
\lesssim
h_T^{1/2} \dual{r_{\partial T}}{v_{\partial T}}_{L^2(\partial T)}.
\end{equation*}
With the piecewise Laplace--Beltrami operator $\Delta_\coarse$ and
$r_\coarse := \phi_\coarse-\alpha f_\coarse+\Delta_\coarse f_\coarse$,
integration by parts then reveals that
\begin{align*}
h_T^{1/2} \norm{r_{\partial T}}{L^2(\partial T)}^2
&\lesssim
h_T^{1/2} \dual{\Delta_\coarse f_\coarse}{v_{\partial T}}_{L^2(\omega_\coarse(T))}
+
h_T^{1/2} \dual{\nabla_\Gamma f_\coarse}{\nabla_\Gamma v_{\partial T}}_{L^2(\omega_\coarse(T))}.
\\
&=
h_T \dual{r_\coarse}{h_T^{-1/2} v_{\partial T}}_{L^2(\omega_\coarse(T))}
-
h_T \dual{-(\phi-\phi_\coarse)+\alpha(f-f_\coarse)}{h_T^{-1/2} v_{\partial T}}_{L^2(\omega_\coarse(T))}
\\
&\quad-
\dual{\nabla_\Gamma (f-f_\coarse)}{h_T^{1/2} \nabla_\Gamma v_{\partial T}}_{L^2(\omega_\coarse(T))},
\end{align*}
and it follows from standard stability estimates for $v_{\partial T}$ that
\begin{equation*}
h_T^{1/2} \norm{r_{\partial T}}{L^2(\partial T)}
\lesssim
h_T^{1/2} \norm{h_\coarse^{1/2}(\phi-\phi_\coarse)}{L^2(\omega_\coarse(T))}
+
(1+\alpha h_T) \norm{f-f_\coarse}{H^1(\omega_\coarse(T))}.
\end{equation*}
Summing over all elements concludes the proof.
\end{proof}

\subsection{Adaptive algorithm}

We abbreviate the weighted-residual error indicators by
\begin{subequations}\label{eq:dirichlet_indirect_mixed:indicators}
\begin{align}
	\eta_{1,\coarse}(T)^2 &:= \norm{h_\coarse^{1/2} \nabla_\Gamma (u - R_1(\phi_\coarse,f_\coarse))}{L^2(T)}^2,
	\\
	\eta_{2,\coarse}(T)^2 &:= h_T^2 \norm{\phi_\coarse - \scal f_\coarse + \Delta_T f_\coarse}{L^2(T)}^2
	+ h_T \norm{[\nabla_\Gamma f_\coarse \cdot \nnu]}{L^2(\partial T)}^2, 
	\\
	\eta_\coarse(T)^2 &:= \eta_{1,\coarse}(T)^2 + \eta_{2,\coarse}(T)^2
	\quad \text{for all } T\in\TT_\coarse.
\end{align}
\end{subequations}
Moreover, we set
\begin{align*}
	\eta_{j,\coarse}(\UU_\coarse)^2 := \sum_{T \in \UU_\coarse} \eta_{j,\coarse}(T)^2, 
	\eta_\coarse(\UU_\coarse)^2 := \eta_{1,\coarse}(\UU_\coarse)^2 + \eta_{2,\coarse}(\UU_\coarse)^2
	\quad \text{for all } j\in\{1,2\}, \UU_\coarse \subseteq \TT_\coarse.
\end{align*}
Note that $\eta_{1,\coarse} = \eta_{1,\coarse}(\TT_\coarse)$, $\eta_{2,\coarse} = \eta_{2,\coarse}(\TT_\coarse)$, and $\eta_\coarse = \eta_\coarse(\TT_\coarse)$. 
As in Section~\ref{sec:dirichlet_indirect_cfie:algorithm}, we can use the weighted-residual error estimator within a standard adaptive algorithm to steer the local mesh refinement of some given sufficiently fine conforming initial triangulation $\TT_0$ such that the discrete inf-sup stability~\eqref{eq:dirichlet_indirect_mixed:discrete_inf-sup} is satisfied.  

\begin{algorithm}\label{alg:dirichlet_indirect_cfie_mixed}
\textbf{Input:} D\"orfler parameter $0 < \theta \le 1$. 
\\
\textbf{Loop:} For each $\ell=0,1,2,\dots$, iterate the following steps:
\begin{enumerate}[\rm(i)]
\item Compute the Galerkin approximation $(\phi_\ell,f_\ell) \in \PP^p(\TT_\ell)\times \mathcal{S}^{p+2}(\TT_\ell)$ by solving~\eqref{eq:dirichlet_indirect_mixed:discretization}.
\item Compute the refinement indicators $\eta_\ell({T})$ from~\eqref{eq:dirichlet_indirect_mixed:indicators} for all elements ${T}\in\TT_\ell$.
\item Determine a minimal set of marked elements $\MM_\ell\subseteq\TT_\ell$ 
satisfying the D\"orfler marking
\begin{align*}
	\theta\,\eta_\ell^2 \le \eta_\ell(\MM_\ell)^2.
\end{align*}%
\item\label{poisson2d:conv:item:abstract_algorithm_refine} Generate refined triangulation $\TT_{\ell+1}:=\refine(\TT_\ell,\MM_\ell)$. %
\end{enumerate}
\textbf{Output:} Refined triangulations $\TT_\ell$, corresponding Galerkin approximations $(\phi_\ell,f_\ell) \in \PP^p(\TT_\ell)\times \mathcal{S}^{p+2}(\TT_\ell)$, and 
weighted-residual error estimators $\eta_\ell$ for all $\ell \in \N_0$.
\end{algorithm}

\subsection{Optimal convergence} 
We define the estimator-based approximation norm $\norm{(\phi,f)}{\mathbb{A}_s^{\rm est}}$ as in~\eqref{eq:linconv_dirichlet_indirect_cfie}, 
and say again that the solution pair $(\phi,f)\in H^{-1/2}(\Gamma)\times H^1(\Gamma)$ lies in the \textit{approximation class $s$ with respect to the estimator} if $\norm{(\phi,f)}{\mathbb{A}_s^{\rm est}} <\infty$. 
The following theorem is the analogon of Theorem~\ref{thm:dirichlet_indirect_cfie}. 
It states that each possible rate $s>0$ is indeed realized by Algorithm~\ref{alg:dirichlet_indirect_cfie}. 
The proof is  given in  Section~\ref{sec:axioms_dirichlet_mixed_cfie}, where we will verify again the axioms of adaptivity from \cite{cfpp14} for the error estimator, which automatically imply the desired result.

\begin{theorem}\label{thm:dirichlet_mixed_cfie}
Suppose that 
the discrete inf-sup stability~\eqref{eq:dirichlet_indirect_mixed:discrete_inf-sup} is satisfied. 
Let $(\TT_\ell)_{\ell\in\N_0}$ be the sequence of triangulations generated by Algorithm~\ref{alg:dirichlet_indirect_cfie}.
Then, for arbitrary $0<\theta\le1$, the estimator converges linearly, i.e., there exist constants $0<\rho_{\rm lin}<1$ and $\const{lin}\ge1$ such that
\begin{align*}
	\eta_{\ell+j}^2\le \const{lin}\rho_{\rm lin}^j\eta_\ell^2\quad\text{for all }j,\ell\in\N_0.
\end{align*}
Moreover, there exists a constant $0<\theta_{\rm opt}\le1$ such that for all $0<\theta<\theta_{\rm opt}$, the estimator converges at optimal rate, i.e., for all $s>0$, there exist constants $c_{\rm opt},\const{opt}>0$ such that
\begin{align*}
	c_{\rm opt}\norm{(\phi,f)}{\mathbb{A}_s^{\rm est}}
	\le \sup_{\ell\in\N_0}{(\# \TT_\ell-\#\TT_0+1)^{s}}\,{\eta_\ell}
	\le \const{opt}\norm{(\phi,f)}{\mathbb{A}_s^{\rm est}}.
\end{align*}
The constant $\theta_{\rm opt}$ depends only on 
$d, k, \Gamma, \alpha, \TT_0$, and $p$, while $\const{lin},\rho_{\rm lin}$ depend additionally on $\theta$, and $\const{opt}$ depends furthermore on $s>0$.
The constant $c_{\rm opt}$ depends only on $d$, $\#\TT_0$,  $s$, and if there exists $\ell_0$ with $\eta_{\ell_0}=0$, then also on $\ell_0$ and $\eta_0$.
\end{theorem}

\subsection{Proof of Theorem~\ref{thm:dirichlet_mixed_cfie}} \label{sec:axioms_dirichlet_mixed_cfie}
The axioms of adaptivity for the estimator from~\cite{cfpp14} now take the following form: 
There exist uniform constants $\const{stab},\const{red},\const{drel},\const{qo}>0$ and $0<\rho_{\rm red}<1$ such that for all $\TT_\coarse\in \T$, $\TT_\fine \in \refine(\TT_\coarse)$ with corresponding patch of all refined elements $\RR_{\coarse,\fine} := \set{T\in\TT_\coarse}{T \cap \bigcup(\TT_\coarse \setminus \TT_\fine) \neq \emptyset}$, the sequence $(\TT_\ell)_{\ell\in\N_0}$ of
Algorithm~\ref{alg:dirichlet_indirect_cfie_mixed}, and all $\ell,n\in\N_0$, there hold:

\begin{enumerate}[(1)]
\renewcommand{\theenumi}{A\arabic{enumi}}
\item Stability on non-refined elements: \label{item:stability2}
	$|\eta_\fine(\TT_\fine \cap \TT_\coarse) - \eta_\coarse(\TT_\coarse \cap \TT_\fine)| 
	\le \const{stab} \norm{(\phi_\fine - \phi_\coarse,f_\fine - f_\coarse)}{H^{-1/2}(\Gamma)\times H^1(\Gamma)}$;
\item Reduction on refined elements: \label{item:reduction2}
	$\eta_\fine(\TT_\fine \setminus \TT_\coarse)
	\le \rho_{\rm red} \eta_\coarse(\TT_\coarse \setminus \TT_\fine)  
		+ \const{red} \norm{(\phi_\fine - \phi_\coarse,f_\fine - f_\coarse)}{H^{-1/2}(\Gamma)\times H^1(\Gamma)}$;
\item General quasi-orthogonality: \label{item:orthogonality2}
	$\sum_{j=\ell}^{\ell + n} \norm{(\phi_{j+1} - \phi_j,f_{j+1} - f_j)}{H^{-1/2}(\Gamma) \times H^1(\Gamma)}^2 
	\le \const{qo}^2 \eta_\ell^2$; 
\item Discrete reliability: \label{item:discrete_reliability2}
	$\norm{(\phi_\fine - \phi_\coarse, f_\fine - f_\coarse)}{H^{-1/2}(\Gamma)\times H^1(\Gamma)}
	\le \const{drel} \eta_\coarse(\RR_{\coarse,\fine}).$
\end{enumerate}

To show the axioms of adaptivity, we combine boundary element arguments of~\cite{bbhp19}, where standard boundary integral equations for the Helmholtz problem are treated, with finite element arguments of \cite{ckns08,bcmn13,bcmmn16}, where the Poisson problem and the Laplace--Beltrami problem are treated.
We also mention that, in contrast to standard FEM, the source term $\phi_\coarse$ in the estimator~\eqref{eq:mixed_est2} is part of the Galerkin approximation itself and is thus mesh-dependent. 
As in Section~\ref{sec:axioms_dirichlet_indirect_cfie}, we again crucially exploit the inverse estimate~\eqref{eq:discrete_invest} and obtain constants that are independent of the sequence of Galerkin approximations $((\phi_\ell,f_\ell))_{\ell\in\N_0}$.

\begin{proof}[Proof of stability on non-refined elements~\eqref{item:stability2}]
Let $\TT_\coarse\in\T$ and $\TT_\fine\in\refine(\TT_\coarse)$. 
By the inverse triangle inequality, it suffices to prove the estimate for $\eta_1$ and $\eta_2$. 
The inverse triangle inequality and the fact that $h_\coarse = h_\fine$ on $\omega := \bigcup (\TT_\coarse \cap \TT_\fine)$ show that 
\begin{align*}
	|\eta_{1,\fine}(\TT_\fine \cap \TT_\coarse) - \eta_{1,\coarse}(\TT_\coarse \cap \TT_\fine)|
	\le \norm{h_\fine^{1/2} \nabla_\Gamma R_1(\phi_\fine - \phi_\coarse,f_\fine - f_\coarse)}{L^2(\Gamma)}.
\end{align*}
The inverse estimates~\eqref{eq:invest_V}--\eqref{eq:invest_K}, and \eqref{eq:discrete_invest} on the mesh $\TT_\fine$ imply that 
\begin{align*}
	\norm{h_\fine^{1/2} \nabla_\Gamma R_1(\phi_\fine - \phi_\coarse, f_\fine - f_\coarse)}{L^2(\Gamma)}
	&\lesssim \norm{\phi_\fine - \phi_\coarse}{H^{-1/2}(\Gamma)} + \norm{f_\fine - f_\coarse}{H^{1/2}(\Gamma)} 
	\\ &\quad
		+ \norm{h_\coarse^{1/2} \nabla_\Gamma (f_\fine-f_\coarse)}{L^2(\Gamma)}
	\\
	&\lesssim \norm{\phi_\fine - \phi_\coarse}{H^{-1/2}(\Gamma)} + \norm{f_\fine - f_\coarse}{H^1(\Gamma)}.
\end{align*}
Similarly, standard inverse estimates and a trace inequality (in the parameter domain) show that
\begin{align*}
	|\eta_{2,\fine}(\TT_\fine \cap \TT_\coarse) - \eta_{2,\coarse}(\TT_\coarse \cap \TT_\fine)|
	&\le \norm{h_\fine(\phi_\fine - \phi_\coarse)}{L^2(\Gamma)} + \norm{f_\fine - f_\coarse}{H^1(\Gamma)} 
	\\
	&\lesssim \norm{\phi_\fine - \phi_\coarse}{H^{-1/2}(\Gamma)} + \norm{f_\fine - f_\coarse}{H^1(\Gamma)}.
\end{align*}
This concludes the proof.
\end{proof}

\begin{proof}[Reduction on refined elements~\eqref{item:reduction2}]
Let $\TT_\coarse\in\T$ and $\TT_\fine\in\refine(\TT_\coarse)$. 
Again, we can estimate $\eta_1$ and $\eta_2$ separately. 
The triangle inequality and the fact that $h_\fine^{1/2} \le \rho h_\coarse^{1/2}$ on $\omega := \bigcup (\TT_\fine \setminus \TT_\coarse) = \bigcup (\TT_\coarse \setminus \TT_\fine)$ for some $0<\rho<1$ show that 
\begin{align*}
	\eta_{1,\fine}(\TT_\fine \setminus \TT_\coarse)
	\le \rho \eta_{1,\coarse}(\TT_\coarse \setminus \TT_\fine)
	+ \norm{h_\fine^{1/2} \nabla_\Gamma R_1(\phi_\fine - \phi_\coarse, f_\fine - f_\coarse)}{L^2(\Gamma)}.
\end{align*}
The second term has already been estimated in the proof of \eqref{item:stability2}. 
Similarly, the fact that $[\nabla_\Gamma f_\coarse \cdot \nnu] = 0$ on edges introduced within the refinement of $\TT_\coarse$ to $\TT_\fine$, standard inverse estimates, and a trace inequality show (as in the proof of \eqref{item:stability2}) for some uniform constants $C, C'>0$ that
\begin{align*}
	\eta_{2,\fine}(\TT_\fine \setminus \TT_\coarse)
	&\le \rho \eta_{2,\coarse}(\TT_\coarse \setminus \TT_\fine) +  C \big(\norm{h_\fine(\phi_\fine - \phi_\coarse)}{L^2(\Gamma)} + \norm{f_\fine - f_\coarse}{H^1(\Gamma)}\big)
	\\
	&\le \rho \eta_{2,\coarse}(\TT_\coarse \setminus \TT_\fine) +  C' \big(\norm{\phi_\fine - \phi_\coarse}{H^{-1/2}(\Gamma)} + \norm{f_\fine - f_\coarse}{H^1(\Gamma)}\big).
\end{align*}
This concludes the proof. 
\end{proof}

\begin{proof}[Proof of general quasi-orthogonality~\eqref{item:orthogonality2}]
By discrete inf-sup stability~\eqref{eq:dirichlet_indirect_mixed:discrete_inf-sup}, \cite{feischl22} is applicable and directly implies this axiom with $\const{qo}=\const{qo}(n)=o(n)$ depending sublinearly on the integer~$n$.
The post-processing argument~\cite[Remark~20]{feischl22} or \cite[Proposition~4.11]{cfpp14} then yield the axiom with a uniform constant~$\const{qo}\simeq 1$. 
\end{proof}

\begin{remark}\label{rem:dirichlet_indirect_cfie_mixed:cea}
As in Remark~\ref{rem:dirichlet_indirect_cfie:cea}, the analysis of \cite{ffp14,bhp17} allows for an alternative proof of~\eqref{item:orthogonality2}, 
and the arguments from there again allow to verify Corollary~\ref{cor:asymptotic2} along the mesh sequence $(\TT_\ell)_{\ell\in\N_0}$ for $\vartheta^{k,\alpha}_{\rm rel,a}(\norm{h_\ell}{L^\infty(\Gamma)})$ and $\vartheta^{k,\alpha}_{\rm eff,a}(\norm{h_\ell}{L^\infty(\Gamma)})$ replaced by some 
$\vartheta^{k,\alpha,\ell}_{\rm rel,a}>0$ and $\vartheta^{k,\alpha,\ell,a}_{\rm eff}>0$ which both converge to $0$ as $\ell\to\infty$.
\end{remark}

\begin{proof}[Proof of discrete reliability~\eqref{item:discrete_reliability2}]
Let $\TT_\coarse\in\T$ and $\TT_\fine\in\refine(\TT_\coarse)$. 
Discrete inf-sup stability~\eqref{eq:dirichlet_indirect_mixed:discrete_inf-sup} and the definition of $(\phi_\fine,f_\fine)$ as well $(\phi_\coarse,f_\coarse)$ (see~\eqref{eq:dirichlet_indirect_mixed:discretization}) show for all $(\psi_\coarse,g_\coarse)\in \PP^p(\TT_\coarse) \times \mathcal{S}^{p+2}(\TT_\coarse)$  that 
\begin{align*}
	\norm{(\phi_\fine - \phi_\coarse,f_\fine - f_\coarse)}{H^{-1/2}(\Gamma) \times H^1(\Gamma)}
	&\lesssim \sup_{\psi_\fine \in \PP^p(\TT_\fine) \atop g_\fine \in \mathcal{S}^{p+2}(\TT_\fine)}  \frac{|\dual{R(\phi_\fine - \phi_\coarse,f_\fine - f_\coarse)}{(\psi_\fine, g_\fine)}_{L^2(\Gamma)}|}{\norm{(\psi_\fine,g_\fine)}{H^{-1/2}(\Gamma)\times H^1(\Gamma)}}
	\\
	& = \sup_{\psi_\fine \in \PP^p(\TT_\fine) \atop g_\fine \in \mathcal{S}^{p+2}(\TT_\fine)}  \frac{|\dual{R(\phi - \phi_\coarse, f - f_\coarse)}{(\psi_\fine - \psi_\coarse, g_\fine - g_\coarse)}_{L^2(\Gamma)}|}{\norm{(\psi_\fine,g_\fine)}{H^{-1/2}(\Gamma)\times H^1(\Gamma)}}.
\end{align*}
Choosing $\psi_\coarse$ as in~\eqref{eq:psi_drel}, we see along the same lines of the non-mixed case that 
\begin{align*}
	|\dual{R_1(\phi - \phi_\coarse, f - f_\coarse)}{\psi_\fine - \psi_\coarse}_{L^2(\Gamma)}| 
	\lesssim \eta_{1,\coarse}(\RR_{\coarse,\fine}) \norm{\psi_\fine}{H^{-1/2}(\Gamma)}.
\end{align*}
Choosing $g_\coarse$ as the Scott--Zhang projection of $g_\fine$, standard arguments as for the Poisson problem show that 
\begin{align*}
	|\dual{R_2(\phi - \phi_\coarse, f - f_\coarse)}{g_\fine - g_\coarse}_{L^2(\Gamma)}|
	\lesssim \eta_{2,\coarse}(\RR_{\coarse,\fine}) \norm{g_\fine}{H^1(\Gamma)};
\end{align*}
cf.~\cite{ckns08,bcmn13,bcmmn16}. 
This concludes the proof.
\end{proof}

\section{Direct integral equations}
\label{sec:direct}

Given again a regularizing operator $M\in \LL(H^{-1/2}(\Gamma), H^{1/2}(\Gamma))$
satisfying~\eqref{eq:compact} and~\eqref{eq:positive}, we consider the representation formula \eqref{eq:representation}
with corresponding Calder\'on system
\begin{subequations}
\begin{align}
	\label{eq:calderon1}
	U|_\Gamma &= (K_k + 1/2)(U|_\Gamma) - V_k(\partial_\nu U),
	\\ \label{eq:calderon2}
	\partial_\nu U &= -W_k(U|_\Gamma) - (K_k' - 1/2)(\partial_\nu U).
\end{align}
\end{subequations} 
We now apply $iM$ to \eqref{eq:calderon2}, add it to \eqref{eq:calderon1}, use the Dirichlet condition $U|_\Gamma = u$, and obtain for $\phi := \partial_\nu U$ that
\begin{align*}
	i M\phi + u = -iMW_k u - iM(K_k' - 1/2)\phi + (K_k + 1/2)u - V_k \phi,
\end{align*}
or equivalently,
\begin{align}\label{eq:dirichlet_direct_cfie}
	V_k \phi + iM(K_k'+1/2)\phi = (K_k - 1/2)u - iMW_k u. 
\end{align}
Clearly, the operator 
\begin{align}
	S:= V_k + iM(K_k'+1/2)\in \LL(H^{-1/2}(\Gamma),H^{1/2}(\Gamma))
\end{align}
on the left-hand side of~\eqref{eq:dirichlet_direct_cfie} is elliptic and symmetric up to the compact perturbation $V_k-V_0^+ + iM(K_k'+1/2)$. 
According to~\cite{bh05}, $S$ is also injective and thus even bijective by the Fredholm alternative. 

\subsection{Non-mixed form}

We can follow the lines of Section~\ref{sec:dirichlet_indirect_cfie} to show reliability and weak efficiency of the weighted-residual estimator $\eta_\coarse:=\norm{h_\coarse^{1/2} \nabla_\Gamma (u - S\phi_\coarse)}{L^2(\Gamma)}$ and optimal convergence of a corresponding adaptive algorithm. 
In particular, the mapping property $K_k'\in \LL(L^2(\Gamma),L^2(\Gamma))$ of~\eqref{eq:mapping_properties_K_pr} and the inverse estimates~\eqref{eq:invest_K_pr} and \eqref{eq:assumption_inverse_new} are exploited. 

\subsection{Mixed form}
\label{sec:direct_mixed}

Similarly as in Section~\ref{sec:dirichlet_indirect_mixed}, we consider $M := (\scal -\Delta_\Gamma)^{-1}$ as mapping from $H^{-1/2}(\Gamma)\subset H^{-1}(\Gamma)$ to $H^1(\Gamma) \subset H^{1/2}(\Gamma)$ and set
\begin{align}
	f := M\big[ (K_k'+1/2)\phi + W_k u\big].
\end{align} 
As a matter of fact, $f=0$ owing to~\eqref{eq:calderon2}. 
We obtain the mixed system 
\begin{align}\label{eq:dirichlet_direct_mixed}
\begin{split}
	V_k \phi + i f &= (K_k - 1/2)u, 
	\\
	- (K_k'+1/2) \phi + \scal f - \Delta_\Gamma f &= W_k u. 
\end{split}
\end{align}
in $H^{1/2}(\Gamma) \times H^{-1}(\Gamma)$ with sought $(\phi,f) \in H^{-1/2}(\Gamma) \times H^1(\Gamma)$. 
Clearly, the operator $R = (R_1,R_2): H^{-1/2}(\Gamma) \times H^1(\Gamma) \to H^{1/2}(\Gamma) \times H^{-1}(\Gamma)$ induced by the left-hand side of~\eqref{eq:dirichlet_direct_mixed} is elliptic up to some compact perturbation. 
Since \eqref{eq:dirichlet_direct_cfie} is uniquely solvable, $R$ is also injective and thus even bijective by the Fredholm alternative. 

It is well-known that if the Dirichlet datum satisfies $u = U|_\Gamma \in H^1(\Gamma)$, then $\phi = \partial_\nu U \in L^2(\Gamma)$; see, e.g., \cite[Theorem 4.24]{mclean00}.
We further recall the mapping properties $K_k\in \LL(H^1(\Gamma), H^1(\Gamma))$ and $W_k\in \LL(H^1(\Gamma), L^2(\Gamma))$ of \eqref{eq:mapping_properties}. 
Thus, we can follow the lines of Section~\ref{sec:dirichlet_indirect_mixed} to show reliability and weak efficiency of 
the weighted-residual estimator $\eta_\coarse = (\eta_{1,\coarse}^2 + \eta_{2,\coarse}^2)^{1/2}$ with 
\begin{align*}
	\eta_{1,\coarse}(T)^2 &:= \norm{h_\coarse^{1/2} \nabla_\Gamma ((K_k - 1/2)u - V_k \phi_\coarse - i f_\coarse)}{L^2(T)}^2,
	\\
	\eta_{2,\coarse}(T)^2 &:= h_T^2 \norm{W_ku + (K_k'+1/2) \phi_\coarse - \scal f_\coarse + \Delta_T f_\coarse}{L^2(T)}^2
	+ h_T \norm{[\nabla_\Gamma f_\coarse \cdot \nnu]}{L^2(\partial T)}^2, 
	\\
	\eta_\coarse(T)^2 &:= \eta_{1,\coarse}(T)^2 + \eta_{2,\coarse}(T)^2
	\quad \text{for all } T\in\TT_\coarse.
\end{align*}
 and optimal convergence of a corresponding adaptive algorithm.

\section{Numerical experiments}
\label{sec:numerics}

We consider the exterior Dirichlet problem~\eqref{eq:helmholtz_strong} with $k^2$ close to an eigenvalue of the interior Dirichlet problem on the following two bounded Lipschitz domains $\Omega$: the circle $$\Omega=\set{x\in\R^2}{|x|<\tfrac{1}{10}}$$ and  the  L-shape $$\Omega={\rm conv} \{(\tfrac{1}{10},0), (0,\tfrac{1}{10}),(0,0),(0,\tfrac{-1}{10})\} \setminus {\rm conv} \{(0,0),(\tfrac{-1}{20},\tfrac{1}{20}),(\tfrac{-1}{10},0),(\tfrac{-1}{20},\tfrac{-1}{20})\}.$$
Note that both domains satisfy $\diam(\Omega)<1$ so that the single-layer operator $V_0$ associated to the Laplace operator is elliptic. 
As exact solution $U$, we prescribe the fundamental solution $U(x_1,x_2):=G_k(x_1,x_2-\tfrac{1}{20})$ with singularity in the interior of $\Omega$. 
We apply the adaptive Algorithm~\ref{alg:dirichlet_indirect_cfie_mixed} with $p=0$ and $\theta\in\{1,0.9\}$ for the mixed forms~\eqref{eq:dirichlet_indirect_mixed} and  \eqref{eq:dirichlet_direct_mixed} of the indirect combined field integral equation~\eqref{eq:dirichlet_indirect_cfie} and the direct combined regularized field integral equation~\eqref{eq:dirichlet_direct_cfie}, respectively.
We select the scaling parameter $\scal = 1$ in the definition of $M$, i.e., $M:=(1-\Delta_\Gamma)^{-1}$.%
\footnote{This choice is dimensionally consistent as $\diam(\Omega) \sim 1$ in our examples.
Setting $\scal = k^2$ is also appealing, but we expect little difference in the results,
since the frequencies considered here are relatively low.}
Note that $\theta=1$ just leads to uniform refinement in each step.
As initial mesh $\TT_0$ for the boundary of the circle, we choose the four arcs $\set{x\in\Gamma}{x_1,x_2\ge 0}$, $\set{x\in\Gamma}{x_1\le 0,x_2\ge0}$, $\set{x\in\Gamma}{x_1,x_2\le 0}$, $\set{x\in\Gamma}{x_1\ge0, x_2\le 0}$ and parametrize them using their arclength parametrization (up to scaling to the reference element $T_{\rm ref}=(0,1)$). 
As initial mesh $\TT_0$ for the boundary of the L-shape, we simply choose the six straight lines that define the boundary and equip them with affine parametrizations. 
We stress that it is in principle unclear whether the initial meshes are indeed sufficiently fine to guarantee the assumed theoretically required discrete inf-sup stability on all refinements of $\TT_0$ (see \eqref{eq:dirichlet_indirect_mixed:discrete_inf-sup} in case of the indirect combined field integral formulation).

For comparison, we also consider the standard direct and indirect integral equations~\eqref{eq:bie_first_kind} and \eqref{eq:bie_first_kind_indirect}, discretized via 
\begin{align}
\label{eq:bie_standard_indirect}
	\dual{V_k\phi_\coarse}{\psi_\coarse}_{L^2(\Gamma)} = \dual{u}{\psi_\coarse}_{L^2(\Gamma)} 
	\quad \text{for all }\psi_\coarse \in \PP^0(\TT_\coarse)
\end{align}
and 
\begin{align*}
	\dual{V_k\phi_\coarse}{\psi_\coarse}_{L^2(\Gamma)} = \dual{(K_k-1/2)u}{\psi_\coarse}_{L^2(\Gamma)}
	\quad \text{for all }\psi_\coarse \in \PP^0(\TT_\coarse),
\end{align*}
respectively, with sought $\phi_\coarse\in\PP^0(\TT_\coarse)$. 
In this case, we employ the weighted-residual estimator $\norm{h_\coarse^{1/2}\nabla_\Gamma V_k(\phi-\phi_\coarse)}{L^2(\Gamma)}$ to steer adaptive refinement.
For $k^2$ not being an eigenvalue of the interior Dirichlet problem, \cite{bbhp19} shows that this estimator is indeed reliable and the algorithm converges at optimal rate. 

In any case, the numerical realization of the adaptive algorithms is based on the MATLAB implementation~\cite{gps22b,gps22a}.
In particular, we employ the built-in ``backslash'' direct solver of MATLAB to solve the linear systems
resulting from the discretization. We highlight that \cite{gps22b} provides a detailed documentation on the stable computation of the Laplace boundary integral operators $V_0, K_0, K_0', W_0$ along with corresponding Galerkin matrices as well as  weighted-residual error estimators. 
Since the Helmholtz kernel $G_k$ has a similar singularity as the Laplace kernel $G_0$, the techniques from \cite{gps22b} directly extend to the Helmholtz boundary integral operators $V_k, K_k, K_k', W_k$. 
Here, we note that, as the hypersingular operator $W_0$ for the Laplace problem, also the hypersingular operator $W_k$ for the Helmholtz problem satisfies an integration-by-parts formula; see, e.g., \cite[page 295]{mclean00}. 

The non-discrete Dirichlet datum $u=U|_\Gamma$ is replaced by its $L^2$-orthogonal projection $u_\coarse$ onto $\mathcal{S}^{2}(\TT_\coarse)$. 
For the Laplace problem, local data oscillation terms that control the additional error $\norm{u-u_\coarse}{H^{1/2}(\Gamma)}$ are investigated in~\cite{ffkmp14}, where it is shown that the algorithms still converge at optimal rate. 
The results of~\cite{ffkmp14} can be readily extended to the considered regularized combined field integral formulations. 
As a matter of fact, to obtain higher-order oscillations, it would suffice to take $u_\coarse\in \mathcal{S}^{1}(\TT_\coarse)$. 
However, since $f$ is discretized in $\mathcal{S}^{2}(\TT_\coarse)$, the required mass matrix is available anyway. 

For the direct approaches, the exact solution $\phi$ is explicitly given as $\phi=\partial_\nu U$. 
The error $\norm{\phi - \phi_\coarse}{V_0}:=\dual{V_0 (\phi - \phi_\coarse)}{\phi - \phi_\coarse}^{1/2} $ (being equivalent to $\norm{\phi - \phi_\coarse}{H^{-1/2}(\Gamma)}$) is approximately computed by replacing $\phi$ with its $L^2$-projection onto $\PP^{1}(\TT_\coarse)$.
We expect that adaptivity (i.e., $\theta=0.9$) yields a convergence rate $\OO((\#\TT_\coarse)^{-3/2})$ for the approximate error and the weighted-residual estimator. 

Concerning the overall runtimes of the adaptive algorithms up to 1000 mesh elements, we observe roughly a factor-of-two difference between the standard and the combined field integral formulation for both the indirect and the direct approach. 
Indeed, the actual time to solve the linear systems is negligible owing to the moderate number of degrees of freedom, so that, 
instead, the bottleneck becomes the assembly of the matrices and the computation of the estimators. 
As a consequence, the observed rough factor of $2$ is not surprising, as the combined field integral equations involve twice as many boundary integral operators as their standard counterparts, 
and the cost of treating the (sparse) matrix associated with the Laplace--Beltrami operator is small compared to the (dense) matrices linked to the integral operators. 
As can be seen below, this translates in some cases in significant
computational savings, since we can obtain the same accuracy in the combined
field formulations with much less degrees of freedom than the standard ones.
 
\newcounter{mycounter}
\newcommand\plotmyfigures[4]{
\foreach \geo in {#1}{
\foreach\exa in {#2}{
\foreach \dir in {#3}{
\foreach \fre in {#4}{
\ifnum\pdfstrcmp{\fre}{critical}=0
	\ifnum\pdfstrcmp{\geo}{circle}=0
		\newcommand\kapset{34.04825558,25.04825558,24.14825558,24.05825558,24.04925558,24.04835558,24.04826558,24.04825658,24.04825558}
	\fi
	\ifnum\pdfstrcmp{\geo}{square}=0
		\newcommand\kapset{32.21441469,23.21441469,22.31441469,22.22441469,22.21541469,22.21451469,22.21442469,22.21441569,22.21441469}
	\fi
	\ifnum\pdfstrcmp{\geo}{lshape}=0
		\newcommand\kapset{72.83185307,63.83185307,62.93185307,62.84185307,62.83285307,62.83195307,62.83186307,62.83185407,62.83185307}
	\fi
\else
	\newcommand\kapset{1,10,100,500,750,1000}
\fi

\begin{figure}[h!]
\centering
\setcounter{mycounter}{0}
\foreach \kap in \kapset{
\stepcounter{mycounter}
\pgfmathtruncatemacro{\result}{2-\value{mycounter}}
\begin{tikzpicture}
\begin{loglogaxis}[
width = 0.33\textwidth, height = 0.33\textwidth, 
ymajorgrids=true, xmajorgrids=true, grid style=dashed, 
tick label style={font=\tiny},
title={\ifnum\pdfstrcmp{\fre}{critical}=0 $k=k_\ifnum\pdfstrcmp{\geo}{circle}=0 O\fi \ifnum\pdfstrcmp{\geo}{lshape}=0 L \fi \ifnum\result=-7\else+10^{\result}\fi$ \else $k=\kap$\fi}
]

\pgfplotstableread[col sep=comma]{data/geo-\geo_dir-\dir_com-0_exa-\exa_kap-\kap_p-0_the-100.csv}\uni
\ifnum\dir=1 \addplot[red, mark=square, dashed, mark size=1.5pt] table[x=N_vec,y=err_vec]{\uni}; \fi
\addplot[blue, mark=x, dashed] table[x=N_vec,y=est_vec]{\uni};
\pgfplotstableread[col sep=comma]{data/geo-\geo_dir-\dir_com-0_exa-\exa_kap-\kap_p-0_the-90.csv}\ada
\ifnum\dir=1\addplot[red, mark=square,mark size=1.5pt] table[x=N_vec,y=err_vec]{\ada}; \fi
\addplot[blue, mark=x] table[x=N_vec,y=est_vec]{\ada};

\pgfplotstableread[col sep=comma]{data/geo-\geo_dir-\dir_com-1_exa-\exa_kap-\kap_p-0_the-100.csv}\unicom
\ifnum\dir=1\addplot[magenta, mark=diamond, dashed] table[x=N_vec,y=err_vec]{\unicom}; \fi
\addplot[cyan, mark=+, dashed] table[x=N_vec,y=est_vec]{\unicom};
\ifnum\pdfstrcmp{\fre}{critical}=0
	\addplot[green, mark=Mercedes star, dashed] table[x=N_vec,y=est2_vec]{\unicom};
\fi
\pgfplotstableread[col sep=comma]{data/geo-\geo_dir-\dir_com-1_exa-\exa_kap-\kap_p-0_the-90.csv}\adacom
\ifnum\dir=1 \addplot[magenta, mark=diamond] table[x=N_vec,y=err_vec]{\adacom}; \fi
\addplot[cyan, mark=+] table[x=N_vec,y=est_vec]{\adacom};
\ifnum\pdfstrcmp{\fre}{critical}=0
	\addplot[green, mark=Mercedes star] table[x=N_vec,y=est2_vec]{\adacom};
\fi

\pgfplotstablegetelem{0}{est_vec}\of{\adacom}
\let\estzero\pgfplotsretval
\pgfplotstablegetelem{0}{N_vec}\of{\adacom}
\let\dofzero\pgfplotsretval
\addplot[black,dashed,domain=\dofzero:1500] {\ifnum\pdfstrcmp{\geo}{lshape}=0 \ifnum\dir=0 3.5* \fi\fi 0.5*\estzero*x^(-3/2)*\dofzero^(3/2)};
\ifnum\pdfstrcmp{\fre}{critical}=0
	\addplot[black,dashed,domain=\dofzero:1500] {\ifnum\pdfstrcmp{\geo}{lshape}=0 \ifnum\dir=0 1* \fi\fi \ifnum\pdfstrcmp{\geo}{circle}=0 \ifnum\dir=1 40* \fi\fi \ifnum\pdfstrcmp{\geo}{lshape}=0 \ifnum\dir=1 80* \fi\fi 0.01*\estzero*x^(-2)*\dofzero^(2)};
\fi
\ifnum\dir=0 \ifnum\pdfstrcmp{\geo}{square}=0 
	\addplot[black,dashed,domain=\dofzero:1500] {0.5*\estzero*x^(-2/3)*\dofzero^(2/3) };
\else \ifnum\pdfstrcmp{\geo}{lshape}=0
	\addplot[black,dashed,domain=\dofzero:1500] {0.5*\estzero*x^(-2/3)*\dofzero^(2/3) };
\fi \fi \fi
\end{loglogaxis}
\end{tikzpicture}%
}    

\ifnum\dir=1
\ifnum\pdfstrcmp{\fre}{critical}=0
\begin{tikzpicture}
\begin{axis}[
hide axis, width = 0.1\textwidth, height = 0.1\textwidth, 
legend style={at={(0,0)},anchor=north, legend columns=4},
legend entries={{sta,uni,err},{sta,ada,err},{com,uni,err},{com,ada,err},{sta,uni,est},{sta,ada,est},{com,uni,est},{com,ada,est},{\phantom{empty}},{\phantom{empty}},{com,uni,est2},{com,ada,est2}}
]
\addlegendimage{red, mark=square, dashed, mark size=1.5pt}
\addlegendimage{red, mark=square,mark size=1.5pt}
\addlegendimage{magenta, mark=diamond, dashed}
\addlegendimage{magenta, mark=diamond}
\addlegendimage{blue, mark=x, dashed}
\addlegendimage{blue, mark=x}
\addlegendimage{cyan, mark=+, dashed}
\addlegendimage{cyan, mark=+}
\ifnum\pdfstrcmp{\fre}{critical}=0 
	\addlegendimage{empty legend}
	\addlegendimage{empty legend}
	\addlegendimage{green, mark=Mercedes star, dashed}
	\addlegendimage{green, mark=Mercedes star}
\fi
\addplot[color=white] coordinates {(0,0)};
\end{axis}
\end{tikzpicture}
\else
\begin{tikzpicture}
\begin{axis}[
hide axis, width = 0.1\textwidth, height = 0.1\textwidth, 
legend style={at={(0,0)},anchor=north, legend columns=4},
legend entries={{sta,uni,err},{sta,ada,err},{com,uni,err},{com,ada,err},{sta,uni,est},{sta,ada,est},{com,uni,est},{com,ada,est}}
]
\addlegendimage{red, mark=square, dashed, mark size=1.5pt}
\addlegendimage{red, mark=square,mark size=1.5pt}
\addlegendimage{magenta, mark=diamond, dashed}
\addlegendimage{magenta, mark=diamond}
\addlegendimage{blue, mark=x, dashed}
\addlegendimage{blue, mark=x}
\addlegendimage{cyan, mark=+, dashed}
\addlegendimage{cyan, mark=+}
\ifnum\pdfstrcmp{\fre}{critical}=0 
	\addlegendimage{empty legend}
	\addlegendimage{empty legend}
	\addlegendimage{green, mark=Mercedes star, dashed}
	\addlegendimage{green, mark=Mercedes star}
\fi
\addplot[color=white] coordinates {(0,0)};
\end{axis}
\end{tikzpicture}
\fi
\else
\begin{tikzpicture}
\begin{axis}[
hide axis, width = 0.1\textwidth, height = 0.1\textwidth, 
legend style={at={(0,0)},anchor=north, legend columns=4},
legend entries={{sta,uni,est},{sta,ada,est},{com,uni,est},{com,ada,est},{\phantom{empty}},{\phantom{empty}},{com,uni,est2},{com,ada,est2}}]
\addlegendimage{blue, mark=x, dashed}
\addlegendimage{blue, mark=x}
\addlegendimage{cyan, mark=+, dashed}
\addlegendimage{cyan, mark=+}
\addlegendimage{empty legend}
\addlegendimage{empty legend}
\addlegendimage{green, mark=Mercedes star, dashed}
\addlegendimage{green, mark=Mercedes star}
\addplot[color=white] coordinates {(0,0)};
\end{axis}
\end{tikzpicture}
\fi 
\caption{\label{fig:\geo\dir\fre}
{\ifnum\pdfstrcmp{\fre}{high}=0 \fi Double-logarithmic convergence plots over number of elements $\#\TT_\ell$ for \ifnum\pdfstrcmp{\geo}{circle}=0 circle\fi \ifnum\pdfstrcmp{\geo}{lshape}=0 L-shape\fi ~and \ifnum\dir=1 direct  \else indirect \fi integral approach. 
\ifnum\dir=1 (Approximate) errors $\norm{\phi-\phi_\ell}{V_0}$ and estimators $\eta_\ell$ \else Estimators $\eta_\ell$ \fi \ifnum\pdfstrcmp{\fre}{critical}=0(and $\eta_{2,\ell}$) \fi are displayed for BEM based on the standard integral equation~\eqref{eq:bie_first_kind} and BEM based on the combined field integral \ifnum\dir=1 equation~\eqref{eq:dirichlet_direct_mixed} \else equation~\eqref{eq:dirichlet_indirect_mixed} \fi under uniform ($\theta=1$) and adaptive ($\theta=0.9$) refinement.
Reference lines are of order \ifnum\pdfstrcmp{\geo}{lshape}=0\ifnum\dir=0 $\OO((\#\TT_\ell)^{-2/3})$, \fi\fi $\OO((\#\TT_\ell)^{-3/2})$ \ifnum\pdfstrcmp{\fre}{critical}=0and $\OO((\#\TT_\ell)^{-2})$\fi.}}
\end{figure}
}
}
}
}
}

\newcommand\plotmyratios[4]{
\foreach \geo in {#1}{
\foreach\exa in {#2}{
\foreach \dir in {#3}{
\foreach \fre in {#4}{
\ifnum\pdfstrcmp{\fre}{critical}=0
	\ifnum\pdfstrcmp{\geo}{circle}=0
		\newcommand\kapset{34.04825558,25.04825558,24.14825558,24.05825558,24.04925558,24.04835558,24.04826558,24.04825658,24.04825558}
	\fi
	\ifnum\pdfstrcmp{\geo}{square}=0
		\newcommand\kapset{32.21441469,23.21441469,22.31441469,22.22441469,22.21541469,22.21451469,22.21442469,22.21441569,22.21441469}
	\fi
	\ifnum\pdfstrcmp{\geo}{lshape}=0
		\newcommand\kapset{72.83185307,63.83185307,62.93185307,62.84185307,62.83285307,62.83195307,62.83186307,62.83185407,62.83185307}
	\fi
\else
	\newcommand\kapset{1,10,100,500,750,1000}
\fi

\begin{figure}[h!]
\centering
\setcounter{mycounter}{0}
\foreach \kap in \kapset{
\stepcounter{mycounter}
\pgfmathtruncatemacro{\result}{2-\value{mycounter}}
\begin{tikzpicture}
\begin{semilogxaxis}[
width = 0.33\textwidth, height = 0.33\textwidth, 
ymajorgrids=true, xmajorgrids=true, grid style=dashed, 
tick label style={font=\tiny},
title={\ifnum\pdfstrcmp{\fre}{critical}=0 $k=k_\ifnum\pdfstrcmp{\geo}{circle}=0 O\fi \ifnum\pdfstrcmp{\geo}{lshape}=0 L \fi \ifnum\result=-7\else+10^{\result}\fi$ \else $k=\kap$\fi}
]

\pgfplotstableread[col sep=comma]{data/geo-\geo_dir-\dir_com-0_exa-\exa_kap-\kap_p-0_the-100.csv}\uni
\addplot[red, mark=square, dashed, mark size=1.5pt] table[x=N_vec,y expr=\thisrow{est_vec}/\thisrow{err_vec}]{\uni};
\pgfplotstableread[col sep=comma]{data/geo-\geo_dir-\dir_com-0_exa-\exa_kap-\kap_p-0_the-90.csv}\ada
\addplot[blue, mark=x] table[x=N_vec,y expr=\thisrow{est_vec}/\thisrow{err_vec}]{\ada};

\pgfplotstableread[col sep=comma]{data/geo-\geo_dir-\dir_com-1_exa-\exa_kap-\kap_p-0_the-100.csv}\unicom
\addplot[magenta, mark=diamond, dashed] table[x=N_vec,y expr=\thisrow{est_vec}/\thisrow{err_vec}]{\unicom};
\pgfplotstableread[col sep=comma]{data/geo-\geo_dir-\dir_com-1_exa-\exa_kap-\kap_p-0_the-90.csv}\adacom
\addplot[cyan, mark=+] table[x=N_vec,y expr=\thisrow{est_vec}/\thisrow{err_vec}]{\adacom}; 

\pgfplotstablegetelem{0}{N_vec}\of{\adacom}
\let\dofzero\pgfplotsretval
\addplot[black,dashed,domain=\dofzero:1500] {0*x+1};

\end{semilogxaxis}
\end{tikzpicture}%
}    

\begin{tikzpicture}
\begin{axis}[
hide axis, width = 0.1\textwidth, height = 0.1\textwidth, 
legend style={at={(0,0)},anchor=north, legend columns=4},
legend entries={{sta,uni},{sta,ada},{com,uni},{com,ada}}
]
\addlegendimage{red, mark=square, dashed, mark size=1.5pt}
\addlegendimage{blue, mark=x}
\addlegendimage{magenta, mark=diamond, dashed}
\addlegendimage{cyan, mark=+}
\addplot[color=white] coordinates {(0,0)};
\end{axis}
\end{tikzpicture}
 
\caption{\label{fig:\geo\dir\fre_ratio}
{\ifnum\pdfstrcmp{\fre}{high}=0 \fi Semi-logarithmic plots of effectivity indices over number of elements $\#\TT_\ell$ for \ifnum\pdfstrcmp{\geo}{circle}=0 circle\fi \ifnum\pdfstrcmp{\geo}{lshape}=0 L-shape\fi ~and \ifnum\dir=1 direct  \else indirect \fi integral approach. 
(Approximate) ratios $\tfrac{\eta_\ell}{\norm{\phi-\phi_\ell}{V_0}}$ are displayed for BEM based on the standard integral equation~\eqref{eq:bie_first_kind} and BEM based on the combined field integral equation~\eqref{eq:dirichlet_direct_mixed} under uniform ($\theta=1$) and adaptive ($\theta=0.9$) refinement. The reference line indicates the ratio $1$.}}
\end{figure}
}
}
}
}
}

\subsection{Experiments on circle}\label{sec:circle}
The eigenvalues $k^2$ of the interior Dirichlet problem on the unit circle are given as the squares of the roots of the Bessel functions of the first kind.
On the scaled circle $$\Omega=\set{x\in\R^2}{|x|<\tfrac{1}{10}},$$ an eigenvalue $k^2$ is thus given as 100 times the square of the first root of $J_0$, i.e., $k_O:=24.04825558...$.
For $k\in \{k_O + 10^1,k_O + 10^0,\dots,k_O + 10^{-6},k_O \}$, we plot in
Figure~\ref{fig:circle0critical} and~\ref{fig:circle1critical} the (approximate)
errors $\norm{\phi-\phi_\ell}{V_0}$ and estimators $\eta_\ell$ for BEM based on
the standard integral equations~\eqref{eq:bie_first_kind_indirect} and \eqref{eq:bie_first_kind}
and BEM based on the combined field integral equations~\eqref{eq:dirichlet_indirect_mixed}
and~\eqref{eq:dirichlet_direct_mixed} under uniform ($\theta=1$) and adaptive ($\theta=0.9$)
refinement.

While for $k$ sufficiently far away from $k_O$, the estimators in Figure~\ref{fig:circle0critical} corresponding to the indirect integral approach converge for each method at optimal rate $\OO((\#\TT_\ell)^{-3/2})$, the involved multiplicative constant deteriorates quickly for standard adaptive BEM as $k$ gets closer to $k_O$. 
For $k=k_O$, no convergence at all is observed. 

In Figure~\ref{fig:circle1critical} corresponding to the direct integral approach, each method yields the optimal rate $\OO((\#\TT_\ell)^{-3/2})$ for the error estimator. 
Unsurprisingly however, although the convergence of the estimator remains unaffected as $k$ gets closer to $k_O$,  the multiplicative constant for the error deteriorates for standard BEM under both uniform and adaptive refinement. 
In particular, the effectivity index, i.e., the ratio between estimator and error, blows up. 
This is not the case for BEM based on combined field integral equations, where error and estimator remain close to each other. 

Recall that for the combined field integral equations, the estimator actually consists of two parts $\eta_{1,\ell}$ and $\eta_{2,\ell}$. 
As observed in Figure~\ref{fig:circle1critical}, the second part $\eta_{2,\ell}$ converges for the indirect and the direct approach at higher rate $\OO((\#\TT_\ell)^{-2})$. 
This higher rate can be motivated by the following arguments: 
As mentioned in the proof of Theorem~\ref{thm:estimator_mixed} and Remark~\ref{rem:reliability_mixed} for the indirect approach, the second part $\eta_{2,\ell}$ is up to potential oscillations equivalent to $\norm{-\phi_\ell + f_\ell - \Delta_\Gamma f_\ell}{H^{-1}(\Gamma)}$ and $\norm{W_ku + (K_k'+1/2) \phi_\ell - f_\ell + \Delta_\Gamma f_\ell}{H^{-1}(\Gamma)}$, respectively. 
The triangle inequality and stability of the involved operators show that both terms are bounded by $\norm{\phi - \phi_\ell}{H^{-1}(\Gamma)}  + \norm{f- f_\ell}{H^{1}(\Gamma)}$. 
From standard approximation theory and $\phi_\ell\in \PP^0(\TT_\ell)$ and $f_\ell\in\mathcal{S}^{2}(\TT_\ell)$, the rate $\eta_{2,\ell} = \OO((\#\TT_\ell)^{-2})$ might indeed be expected.

\plotmyfigures{circle}{fund}{0}{critical}
\plotmyfigures{circle}{fund}{1}{critical}


%

\subsection{Experiments on L-shape}\label{sec:lshape}
An exact eigenvalue $k^2$ of the interior Dirichlet problem on the L-shape $$\Omega={\rm conv} \{(\tfrac{1}{10},0), (0,\tfrac{1}{10}),(0,0),(0,\tfrac{-1}{10})\} \setminus {\rm conv} \{(0,0),(\tfrac{-1}{20},\tfrac{1}{20}),(\tfrac{-1}{10},0),(\tfrac{-1}{20},\tfrac{-1}{20})\}$$ is given for $k_L:=20\pi$.
For $k\in \{k_L + 10^1,k_L + 10^0,\dots,k_L + 10^{-6},k_L \}$, we plot in Figure~\ref{fig:lshape0critical} and~\ref{fig:lshape1critical} the (approximate) errors $\norm{\phi-\phi_\ell}{V_0}$ and estimators $\eta_\ell$ for BEM based on the standard integral
\eqref{eq:bie_first_kind_indirect} and \eqref{eq:bie_first_kind} and BEM based on the combined field integral equations~\eqref{eq:dirichlet_indirect_mixed} and~\eqref{eq:dirichlet_direct_mixed} under uniform ($\theta=1$) and adaptive ($\theta=0.9$) refinement.

For $k=k_L+10$, the estimators in Figure~\ref{fig:lshape0critical} corresponding to the indirect integral approach converge at the suboptimal rate $\OO((\#\TT_\ell)^{-2/3})$ under uniform refinement.
Adaptive refinement regains the optimal rate $\OO((\#\TT_\ell)^{-3/2})$. 
The same holds for BEM based on combined field integral equations as $k$ gets closer to $k_L$. 
Instead, for standard BEM, we even observe a strong growth of the estimator in the increasing preasymptotic phase.
The seemingly optimal convergence under uniform refinement for $k\neq k_L+10$ also appears to be a preasymptotic behavior as is  suggested by the plot for $k=k_L+10$ (and other values between $k_L+1$ and $k_L+10$ which are not displayed).

In Figure~\ref{fig:lshape1critical} corresponding to the direct integral approach, we observe a similar behavior as for the circle considered in Section~\ref{sec:circle}. 
We also mention that the second part $\eta_{2,\ell}$ again converges for the indirect and the direct approach at higher rate $\OO((\#\TT_\ell)^{-2})$ even for uniform refinement. 

\plotmyfigures{lshape}{fund}{0}{critical}
\plotmyfigures{lshape}{fund}{1}{critical}

\subsection{Experiments with high frequencies}

Finally, in the setting of Section~\ref{sec:circle} and~\ref{sec:lshape}, we consider high
frequencies $k$ (which do not correspond to eigenvalues of the interior Dirichlet
problem) on the circle and the L-shape. For $k\in \{1, 10, 100,$ $500, 750, 1000\}$, we plot
in Figure~\ref{fig:circle1high} and \ref{fig:lshape1high} the (approximate) errors
$\norm{\phi-\phi_\ell}{V_0}$ and estimators $\eta_\ell$ for BEM based on the direct
standard integral equation~\eqref{eq:bie_first_kind} and BEM based on the direct combined
field integral equation~\eqref{eq:dirichlet_direct_mixed} under uniform ($\theta=1$) and
adaptive ($\theta=0.9$) refinement. The estimators and the errors converge in each
case at the optimal rate $\OO((\#\TT_\ell)^{-3/2})$.
We also display in Figure~\ref{fig:circle1high_ratio} and \ref{fig:lshape1high_ratio},
the corresponding effectivity indices $\tfrac{\eta_\ell}{\norm{\phi-\phi_\ell}{V_0}}$.
As expected from Corollary~\ref{cor:asymptotic2} and
Remark~\ref{rem:dirichlet_indirect_cfie_mixed:cea} (for the indirect approach), the
effectivity indices are asymptotically robust with respect to $k$.

\plotmyfigures{circle}{fund}{1}{high}
\plotmyratios{circle}{fund}{1}{high}
\plotmyfigures{lshape}{fund}{1}{high}
\plotmyratios{lshape}{fund}{1}{high}

\section*{Acknowledgement}

GG  acknowledges funding by the Deutsche Forschungsgemeinschaft (DFG, German Research Foundation) under Germany's Excellence Strategy – EXC-2047/1 – 390685813.
TCF acknowledges funding by the ANR JCJC project APOWA (research grant ANR-23-CE40-0019-01).

\bibliographystyle{abbrv}
\bibliography{literature}

\end{document}